\documentclass[10pt]{amsart}
\usepackage{amssymb, amsthm, amsmath}
\usepackage[all]{xy}
\usepackage{graphicx}
\usepackage{psfrag}

\newtheorem{prop}{Proposition}
\newtheorem{thm}{Theorem}
\newtheorem{cor}{Corollary}
\newtheorem{lemma}{Lemma}
\theoremstyle{definition}

\newtheorem{defn}{Definition}
\newtheorem{example}{Example}
\newtheorem{remark}{Remark}

\newcommand\A{{\mathcal A}}

\newcommand{\M}{\mathcal M}
\newcommand\C{{\mathbb C}}
\newcommand\Q{{\mathbb Q}}
\newcommand\N{{\mathbb N}}

\newcommand{\ti}{\vartheta}

\newcommand{\Ti}{\Theta}

\newcommand{\om}{{\varpi}}

\newcommand\cC{{\mathcal C}}

\newcommand\Z{{\mathbb Z}}
\newcommand\cP{{\mathcal P}}

\newcommand\AS{{\mathfrak S}}
\newcommand\BS{{\mathfrak B}}
\newcommand\CS{{\mathfrak C}}
\newcommand\DS{{\mathfrak D}}

\newcommand\IGam{{\mathrm I}\Gamma}
\newcommand\ILam{{\mathrm I}\Lambda}
\newcommand\IB{{\mathrm I}B}
\newcommand\fraka{{\mathfrak a}}

\newcommand\RR{R^{\infty}}
\newcommand\al{\alpha}
\newcommand{\be}{\beta}
\newcommand\la{\lambda}

\newcommand\supp{{\mathrm{supp}}}
\newcommand\Eta{H}

\newcommand\ssm{\smallsetminus}
\newcommand\noin{\noindent}

\newcommand\bull{{\scriptscriptstyle \bullet}}
\newcommand\eqto{\stackrel{\lower1.5pt\hbox{$\scriptstyle\sim\,$}}\to}
\newcommand\ov{\overline}
\newcommand\hra{\hookrightarrow}
\newcommand\wh{\widehat}
\newcommand\wt{\widetilde}

\DeclareMathOperator{\Ker}{Ker}

\DeclareMathOperator{\Sp}{Sp}
\DeclareMathOperator{\SO}{SO}

\DeclareMathOperator{\GL}{GL}

\DeclareMathOperator{\HH}{\mathrm{H}}

\begin{document}

\title[Schubert polynomials, theta and eta polynomials, and $W$-invariants]
{Schubert polynomials, theta and eta polynomials, and Weyl group invariants}

\date{September 13, 2019}

\author{Harry~Tamvakis} 
\address{University of Maryland, Department of
Mathematics, William E. Kirwan Hall, 4176 Campus Drive, 
College Park, MD 20742, USA}
\email{harryt@math.umd.edu}

\dedicatory{In memory of Alain Lascoux}



\thanks{The author was supported in part by NSF Grant DMS-1303352.}

\begin{abstract}
We examine the relationship between the (double) Schubert polynomials
of Billey-Haiman and Ikeda-Mihalcea-Naruse and the (double) theta and
eta polynomials of Buch-Kresch-Tamvakis and Wilson from the
perspective of Weyl group invariants. We obtain generators for the
kernel of the natural map from the corresponding ring of Schubert
polynomials to the (equivariant) cohomology ring of symplectic and
orthogonal flag manifolds.
\end{abstract}

\maketitle

\section{Introduction}

The theory of Schubert polynomials due to Lascoux and Sch\"utzenberger
\cite{LS1} provides a canonical set of polynomial representatives for
the Schubert classes on complete type A flag manifolds. The classical
Schur polynomials are identified with the Schubert polynomials
representing the classes which pull back from Grassmannians. There are
natural analogues of these objects for the symplectic and orthogonal
groups: the Schubert polynomials of Billey and Haiman \cite{BH}, and
the theta and eta polynomials of Buch, Kresch, and the author
\cite{BKT1, BKT2}, respectively. We also have `double' versions of 
the aforementioned polynomials, which represent the Schubert classes
in the torus-equivariant cohomology ring, and in the setting of
degeneracy loci of vector bundles (see \cite{L2, F1, IMN1} and
\cite{KL, L1, W, IM, TW, T6}, respectively).

The goal of this work is to study the relation between these two
families of polynomials from the point of view of Weyl group
symmetries, following the program set out in \cite{LS1, LS2, M2} in
Lie type A.  The key observation is that the theta and eta polynomials
of a fixed level $n$ form a basis of the Weyl group invariants in the
associated ring of Schubert polynomials (Propositions \ref{invlem} and
\ref{invlemD}).  In this introduction, for simplicity, we review the
story in type A, and describe its analogue in type C, in the case of
`single' polynomials, leaving the extensions to the `double' case and
the orthogonal Lie types B and D to the main body of the paper.

Let $S_\infty:=\cup_kS_k$ be the group of permutations of the positive
integers which leave all but a finite number of them fixed. For any
$n\geq 1$, let $S^{(n)}$ denote the set of those permutations
$\om=(\om_1,\om_2,\ldots)$ in $S_\infty$ such that $\om_{n+1} <
\om_{n+2} < \cdots$.  If $X_n:=(x_1,\ldots,x_n)$ is a family of $n$
commuting independent variables, then the single Schubert polynomials
$\AS_\om$ of Lascoux and Sch\"utzenberger \cite{LS1}, as $\om$ ranges
over $S^{(n)}$, form a $\Z$-basis of the polynomial ring $\Z[X_n]$.
If $M_n:=\GL_n/B$ denotes the complete type A flag manifold over $\C$,
then there is a surjective ring homomorphism $\rho_n:\Z[X_n]\to
\HH^*(M_n)$ which maps the polynomial $\AS_\om$ to the cohomology
class $[X_\om]$ of a codimension $\ell(\om)$ Schubert variety $X_\om$
in $M_n$, if $\om\in S_n$, and to zero, otherwise.

The Weyl group $S_n$ acts on
$\Z[X_n]$ by permuting the variables, and the subring 
$\Z[X_n]^{S_n}$ of $S_n$-invariants is the ring $\Lambda_n$ of 
symmetric polynomials in $x_1,\ldots,x_n$. We have $\Lambda_n=
\Z[e_1(X_n),\ldots,e_n(X_n)]$, where $e_i(X_n)$ denotes the $i$-th
elementary symmetric polynomial. The kernel of $\rho_n$
is the ideal $\ILam_n$ of $\Z[X_n]$ generated by the homogeneous
elements of positive degree in $\Lambda_n$. We therefore have
\[
\ILam_n=\bigoplus_{\om\in S^{(n)}\ssm S_n} \Z\,\AS_\om = 
\langle e_1(X_n),\ldots,e_n(X_n)\rangle
\]
and recover the Borel presentation \cite{Bo} of the cohomology ring
\[
\HH^*(\GL_n/B) \cong \Z[X_n]/\ILam_n.
\]

Any Schubert polynomial $\AS_\om$ which lies in $\Lambda_n$ is equal
to a Schur polynomial $s_\la(X_n)$ indexed by a partition
$\la=(\la_1,\ldots,\la_n)$ associated to the Grassmannian permutation
$\om$, and these elements form a $\Z$-basis of $\Lambda_n$. One knows
that
\begin{equation}
\label{StoS}
s_\la(X_n) 
= \left. A(x^{\la+\delta_{n-1}})\right\slash A(x^{\delta_{n-1}})
\end{equation}
where $A:= \sum_{\om\in S_n}(-1)^{\ell(\om)}\om$ is the alternating
operator on $\Z[X_n]$, $x^\al$ denotes
$x_1^{\al_1}\cdots x_n^{\al_n}$ for any integer vector exponent
$\al$, and $\delta_k:=(k,\ldots,1,0)$ for every $k\geq 0$. Equation
(\ref{StoS}) may be identified with the Weyl character formula for
$\GL_n$. Alternatively, one has the (dual) Jacobi-Trudi identity
\begin{equation}
\label{JT}
s_\la(X_n) = \det(e_{\la'_i+j-i}(X_n))_{i,j} = 
\prod_{i<j} (1-R_{ij}) \,e_{\la'}(X_n),
\end{equation}
where $\la'$ is the conjugate partition of $\la$,
$e_{\nu} := e_{\nu_1}e_{\nu_2}\cdots$ for any integer vector $\nu$,
and the $R_{ij}$ are Young's raising operators, with 
$R_{ij}e_\nu:=e_{R_{ij}\nu}$ (see \cite[I.3]{M1}).

If $\om_0$ denotes the longest permutation in $S_n$, then the divided
difference operator $\partial_{\om_0}$ gives a $\Lambda_n$-linear map
$\Z[X_n]\to \Lambda_n$, and the equation $\langle f,g\rangle =
\partial_{\om_0}(fg)$ defines a scalar product on $\Z[X_n]$, with
values in $\Lambda_n$. The Schubert polynomials $\{\AS_\om\}_{\om\in
S_n}$ form a basis for $\Z[X_n]$ as a $\Lambda_n$-module, and
satisfy an orthogonality property under this product, which
corresponds to the natural duality pairing on $\HH^*(M_n)$.

The above narrative admits an exact analogue for the symplectic
group.  Let $c:=(c_1,c_2,\ldots)$ be a sequence of commuting
variables, and set $c_0:=1$ and $c_p=0$ for $p<0$. Consider the graded
ring $\Gamma$ which is the quotient of the polynomial ring $\Z[c]$
modulo the ideal generated by the relations
\begin{equation}
\label{basicrels}
c_p^2+2\sum_{i=1}^p(-1)^ic_{p+i}c_{p-i}=0, \ \ \ \text{for all $p\geq 1$}.
\end{equation}
The ring $\Gamma$ is isomorphic to the ring of Schur $Q$-functions
\cite[III.8]{M1} and to the stable cohomology ring of the Lagrangian
Grassmannian, following \cite{P, J}.

Let $W_k$ denote the hyperoctahedral group of signed permutations on
the set $\{1,\ldots,k\}$. For each $k\geq 1$, we embed $W_k$ in
$W_{k+1}$ by adjoining the fixed point $k+1$, and set $W_\infty
:=\cup_k W_k$. For any $n\geq 0$, let $W^{(n)}$ denote the set of
those elements $w=(w_1,w_2,\ldots)$ in $W_\infty$ such that $w_{n+1} <
w_{n+2} < \cdots$. The type C single Schubert polynomials $\CS_w$ of
Billey and Haiman \cite{BH}, as $w$ ranges over $W^{(n)}$, form a
$\Z$-basis of the ring $\Gamma[X_n]$ (Proposition 
\ref{basisprop})\footnote{The Billey-Haiman Schubert polynomials are actually 
power series; see Section \ref{Cprelims}}. If
$\M_n:=\Sp_{2n}/B$ denotes the complete symplectic flag manifold over
$\C$, then there is a surjective ring homomorphism
$\pi_n:\Gamma[X_n]\to \HH^*(\M_n)$ which maps the polynomial $\CS_w$
to the class $[X_w]$ of a codimension $\ell(w)$ Schubert variety $X_w$
in $\M_n$, if $w\in W_n$, and to zero, otherwise.

There is a natural action of the Weyl group $W_n$ on $\Gamma[X_n]$
which extends the $S_n$ action on $\Z[X_n]$ (see \S \ref{Cprelims}).  
The subring $\Gamma[X_n]^{W_n}$ of $W_n$-invariants is
the ring $\Gamma^{(n)}$ of {\em theta polynomials of level $n$}
(Proposition \ref{invlem}). The ring $\Gamma^{(n)}$ was defined in
\cite[\S 5.1]{BKT1} as $\Gamma^{(n)} := \Z[{}^nc_1, {}^nc_2, \ldots]$,
where
\[
{}^nc_p:=\sum_{j=0}^p c_{p-j}e_j(X_n), \ \ \ \text{for $p\geq 1$}.
\]
The kernel of $\pi_n$ is the ideal $\IGam^{(n)}$ of $\Gamma[X_n]$
generated by the homogeneous elements of positive degree in $\Gamma^{(n)}$.
We therefore have
\begin{equation}
\label{IGtoc}
\IGam^{(n)}=\bigoplus_{w\in W^{(n)}\ssm W_n} \Z\,\CS_w =
\langle {}^nc_1, {}^nc_2, \ldots\rangle
\end{equation}
and obtain (Corollary \ref{prescor}) a canonical isomorphism
\[
\HH^*(\Sp_{2n}/B) \cong \Gamma[X_n]/\IGam^{(n)}.
\]

Following \cite{BKT1}, any Schubert polynomial $\CS_w$ which lies in
$\Gamma^{(n)}$ is equal to a {\em theta polynomial} ${}^n\Ti_\la$
indexed by an $n$-strict partition $\la=\la(w)$ associated to the
$n$-Grassmannian element $w$, and these polynomials form a $\Z$-basis
of $\Gamma^{(n)}$. The polynomial ${}^n\Ti_\la$ was defined in
\cite{BKT1} using the raising operator formula
\begin{equation}
\label{raiseC}
{}^n\Ti_\la := \prod_{i<j} (1-R_{ij}) \prod_{(i,j)\in \cC(\la)}
(1+R_{ij})^{-1}\, ({}^nc)_\la,
\end{equation}
where $({}^nc)_\nu:={}^nc_{\nu_1}{}^nc_{\nu_2}\cdots$ and $\cC(\la)$
denotes the set of pairs $(i,j)$ with $i<j$ and
$\la_i+\la_j>2n+j-i$. This is the symplectic version of equation
(\ref{JT}).

There is also a symplectic analogue of formula (\ref{StoS}). Let $w_0$
denote the longest element of $W_n$, define $\wh{w}:= ww_0$ and
consider the {\em multi-Schur Pfaffian}
\begin{equation}
\label{SchPf}
{}^{\nu(\wh{w})}Q_{\la(\wh{w})} := \prod_{i<j} \frac{1-R_{ij}}{1+R_{ij}}
\,{}^{\nu(\wh{w})}c_{\la(\wh{w})}
\end{equation}
where $\nu(\wh{w})$ and $\la(\wh{w})$ are certain partitions
associated to $\wh{w}$ (see (\ref{Qdef}) and Definition
\ref{codedef}). We then have (Theorem \ref{altCthm})
\begin{align}
\label{TtoQ}
{}^n\Ti_{\la(w)} 
=(-1)^{n(n+1)/2}\left. \A\left({}^{\nu(\wh{w})}Q_{\la(\wh{w})}\right)\right\slash
\A\left(x^{\delta_n+\delta_{n-1}}\right)
\end{align}
where $\A:= \sum_{w\in W_n}(-1)^{\ell(w)}w$ is the alternating
operator on $\Gamma[X_n]$. In the special case when $w\in S_\infty$,
with $\la=\la(w)$, equation (\ref{TtoQ}) becomes
\[
{}^n\Ti_\la =
(-1)^{n(n+1)/2}\left. \A\left({}^{\delta_{n-1}}Q_{\delta_n+\delta_{n-1}+\la'}
\right)\right\slash \A\left(x^{\delta_n+\delta_{n-1}}\right).
\]

The maximal divided difference operator $\partial_{w_0}$ gives a
$\Gamma^{(n)}$-linear map $\Gamma[X_n]\to \Gamma^{(n)}$, and the
equation $\langle f,g\rangle = \partial_{w_0}(fg)$ defines a scalar
product on $\Gamma[X_n]$, with values in $\Gamma^{(n)}$. The Schubert
polynomials $\{\CS_w\}_{w\in W_n}$ form a basis for $\Gamma[X_n]$ as a
$\Gamma^{(n)}$-module (Corollary \ref{basiscor}), and satisfy an
orthogonality property under this product (Proposition
\ref{prodprop}), which corresponds to the duality pairing on
$\HH^*(\M_n)$. A similar scalar product in the finite case 
was introduced and studied in \cite{LP1}.

As mentioned earlier, we provide analogues of most of the above facts
for the double Schubert, theta, and eta polynomials. Our main new
results (Theorems \ref{Cthm} and \ref{Dthm}) are the double versions
of equation (\ref{IGtoc}), which exhibit natural generators for the
kernel of the {\em geometrization map} of \cite[\S 10]{IMN1} from the
(stable) ring of double Schubert polynomials to the equivariant
cohomology ring of the corresponding (finite dimensional) symplectic
or orthogonal flag manifold. This is done by using an idea from
\cite[Lemma 1]{T1} together with the transition equations of \cite{B,
  IMN1} to write the Schubert polynomials in this kernel as an
explicit linear combination of these generators, which is important in
applications. The double versions of formula (\ref{TtoQ}) rely on the
equality of the multi-Schur Pfaffian (\ref{SchPf}) -- and its
orthogonal analogue -- with certain double Schubert polynomials
(Propositions \ref{PfCprop} and \ref{PfDprop}). This latter fact is
an extension of \cite[Thm.\ 1.2]{IMN1}, which may be deduced from the
(even more general) Pfaffian formulas of Anderson and Fulton
\cite{AF}. We give an independent treatment here, using the right
divided difference operators.

This paper is organized as follows. In Section \ref{dspC} we recall
the type C double Schubert polynomials and the geometrization map
$\pi_n$ from $\Gamma[X_n,Y_n]$ to the equivariant cohomology ring of
$\Sp_{2n}/B$, and obtain canonical generators for the kernel of
$\pi_n$. In Section \ref{ddstheta} we define the statistics $\nu(w)$
and $\la(w)$ of a signed permutation $w$ and prove the analogue of
formula (\ref{TtoQ}) for double theta polynomials. Section \ref{sspC}
examines some related facts about single type C Schubert polynomials,
including the scalar product with values in the ring $\Gamma^{(n)}$ of
$W_n$-invariants. Sections \ref{dspBD}, \ref{ddsdbleta} and \ref{sspD}
study the corresponding questions in the orthogonal Lie types B and D.

I dedicate this article to the memory of Alain Lascoux, whose warm
personality and vision about symmetric functions and Schubert
polynomials initially assisted, and subsequently inspired my research,
from its beginning to the present day.

\section{Double Schubert polynomials of type C}
\label{dspC}

\subsection{Preliminaries}
\label{Cprelims}
We recall the type C double Schubert polynomials of Ikeda, Mihalcea,
and Naruse \cite{IMN1}, employing the notational conventions of the
introduction, which are similar to those used in \cite{AF}. These
differ from the Schubert polynomials found in \cite{BH,IMN1} and our
papers \cite{T4, T5, T7}, in that the ring $\Gamma$ is realized using
the generators $c_p$ and relations (\ref{basicrels}) among them,
instead of the formal power series known as Schur $Q$-functions, which
are not required in the present work.  The connection between these
power series and the Schubert polynomials used here was first
explained in \cite{T2, T3} (in the case of single polynomials) and
\cite{IMN1, T4} (for their double versions). We refer to \cite[Section
  7.3]{T5} and \cite[Section 5]{T7} for a detailed account of this
history.

Let $X:=(x_1,x_2,\ldots)$ and $Y:=(y_1,y_2,\ldots)$ be two lists of
commuting independent variables, and set $X_n:=(x_1,\ldots,x_n)$ and
$Y_n:=(y_1,\ldots,y_n)$ for each $n\geq 1$. The Weyl group for the
root system of type $\text{C}_n$ is the group of signed permutations
on the set $\{1,\ldots,n\}$, denoted $W_n$. We write the elements of
$W_n$ as $n$-tuples $(w_1,\ldots,w_n)$, where $w_i:=w(i)$ for $1\leq i
\leq n$.  The group $W_\infty=\cup_k W_k$ is generated by the simple
transpositions $s_i=(i,i+1)$ for $i\geq 1$ together with the sign
change $s_0$, which fixes all $j\geq 2$ and sends $1$ to $\ov{1}$ (a
bar over an integer here means a negative sign). For $w\in W_\infty$,
we denote by $\ell(w)$ the {\em length} of $w$, which is the least
integer $\ell$ such that we can write $w=s_{i_1}\cdots s_{i_\ell}$ for
some indices $i_j\geq 0$.

There is an action of $W_\infty$ on $\Gamma[X,Y]$ by ring
automorphisms, defined as follows. The simple reflections $s_i$ for $i\geq 1$ 
act by interchanging $x_i$ and $x_{i+1}$ while leaving all the remaining
variables fixed. The reflection $s_0$ maps $x_1$
to $-x_1$, fixes the $x_j$ for $j\geq 2$ and all the $y_j$, and
satisfies
\begin{equation}
\label{equiva}
s_0(c_p) := c_p+2\sum_{j=1}^p x_1^jc_{p-j} \ \ \text{for all $p\geq 1$}.
\end{equation}

For each $i\geq 0$, define the {\em divided difference operator}
$\partial_i^x$ on $\Gamma[X,Y]$ by
\[
\partial_0^xf := \frac{f-s_0f}{-2x_1}, \qquad
\partial_i^xf := \frac{f-s_if}{x_i-x_{i+1}} \ \ \ \text{for $i\geq 1$}.
\]
Consider the ring involution $\omega:\Gamma[X,Y]\to\Gamma[X,Y]$
determined by
\[
\omega(x_j) = -y_j, \qquad
\omega(y_j) = -x_j, \qquad
\omega(c_p)=c_p
\]
and set $\partial_i^y:=\omega\partial_i^x\omega$ for each $i\geq 0$. 

The double Schubert polynomials $\CS_w=\CS_w(X,Y)$ for $w\in
W_{\infty}$ are the unique family of elements of $\Gamma[X,Y]$ such
that
\begin{equation}
\label{ddCeqinit}
\partial_i^x\CS_w = \begin{cases}
\CS_{ws_i} & \text{if $\ell(ws_i)<\ell(w)$}, \\ 
0 & \text{otherwise},
\end{cases}
\quad
\partial_i^y\CS_w = \begin{cases}
\CS_{s_iw} & \text{if $\ell(s_iw)<\ell(w)$}, \\ 
0 & \text{otherwise},
\end{cases}
\end{equation}
for all $i\geq 0$, together with the condition that the constant term
of $\CS_w$ is $1$ if $w=1$, and $0$ otherwise.  For any $w\in
W_\infty$, the corresponding (single) Billey-Haiman Schubert
polynomial of type C is $\CS_w(X):=\CS_w(X,0)$.  It is known that the
$\CS_w(X)$ for $w\in W_\infty$ form a $\Z$-basis of
$\Gamma[X]=\Gamma[x_1,x_2,\ldots]$, and the $\CS_w(X,Y)$ for $w\in
W_\infty$ form a $\Z[Y]$-basis of
$\Gamma[X,Y]=\Gamma[x_1,y_1,x_2,y_2,\ldots]$. See \cite{IMN1} for
further details, noting that the polynomial called $\CS_w(z,t;\, x)$
in op.\ cit., which is a formal power series in the $x$
variables, would be the polynomial denoted by $\CS_w(z,t)$ here.

In the sequel, for every $i\geq 0$, we set $\partial_i:=\partial_i^x$.
For any $w\in W_\infty$, we define a divided difference operator
$\partial_w:=\partial_{i_1}\circ \cdots \circ \partial_{i_\ell}$, for
any choice of indices $(i_1,\ldots,i_\ell)$ such that $w=
s_{i_1}\cdots s_{i_\ell}$ and $\ell=\ell(w)$.  According to
\cite[Prop.\ 8.4]{IMN1}, for any $u,w\in W_\infty$, we have
\begin{equation}
\label{Cstar}
\partial_u\CS_w(X,Y) = \begin{cases}
\CS_{wu^{-1}}(X,Y) & \text{if $\ell(wu^{-1}) = \ell(w) - \ell(u)$}, \\
0 & \text{otherwise}.
\end{cases}
\end{equation}

\subsection{The set $W^{(n)}$ and the ring $\Gamma[X_n,Y_n]$}
For every $n\geq 1$, let 
\[
W^{(n)}:=\{w\in W_\infty\ |\ w_{n+1}< w_{n+2}<\cdots\}.
\] 

\begin{prop}
\label{basisprop}
The $\CS_w(X)$ for $w\in W^{(n)}$ form a $\Z$-basis of $\Gamma[X_n]$.
\end{prop}
\begin{proof}
We have that $\CS_w(X)\in \Gamma[X_n]$ if and only if
$\partial_m\CS_w(X)=0$ for all $m>n$ if and only if $w\in
W^{(n)}$. Suppose that $f\in \Gamma[X_n]$ is a polynomial which is not
in the $\Z$-span of the $\CS_w(X)$, $w\in W^{(n)}$. Then $f$ can be
written as an integer linear combination of Schubert polynomials
\begin{equation}
\label{feq}
f(X) = \sum_w e_w\CS_w(X)
\end{equation}
where there is at least one $w$ with $e_w\neq 0$ and $w\notin
W^{(n)}$. Hence for some $m>n$ we have $\partial_m\CS_w =
\CS_{ws_m}$, and since $\partial_mf=0$, we obtain from (\ref{feq}) a
nontrivial linear dependence relation among the Schubert polynomials,
which is a contradiction. This proves that the $\CS_w(X)$ for $w\in
W^{(n)}$ span $\Gamma[X_n]$, and therefore the result.
\end{proof}

\begin{prop}
\label{basisprop2}
The $\CS_w(X,Y)$ for $w\in W^{(n)}$ form a $\Z[Y]$-basis of $\Gamma[X_n,Y]$.
\end{prop}
\begin{proof}
The $\CS_w(X,Y)$ for $w\in W_\infty$ are linearly independent over $\Z[Y]$.
By Proposition \ref{basisprop} we know that the $\CS_w(X)$ for 
$w\in W^{(n)}$ form a $\Z$-basis of $\Gamma[X_n]$. According to 
\cite[Cor.\ 8.10]{IMN1}, we have
\[
\CS_w(X,Y) = \sum_{uv=w} \AS_{u^{-1}}(-Y)\CS_v(X)
\]
summed over all factorizations $uv=w$ with $\ell(u)+\ell(v)=\ell(w)$
and $u\in S_\infty$. Since the term of lowest $y$-degree in the sum is
$\CS_w(X)$, the proposition follows.
\end{proof}

Let $\CS^{(n)}_w=\CS^{(n)}_w(X_n,Y_n)$ be the polynomial obtained from $\CS_w(X,Y)$ by
setting $x_j=y_j=0$ for all $j>n$.

\begin{cor}
\label{basiscor2}
The $\CS^{(n)}_w$ for $w\in W^{(n)}$ form a $\Z[Y_n]$-basis of
$\Gamma[X_n,Y_n]$.
\end{cor}

\subsection{The geometrization map $\pi_n$}
\label{geommap}
The double Schubert polynomials $\CS^{(n)}_w(X,Y)$ for $w\in W_n$ represent
the equivariant Schubert classes on the symplectic flag manifold.  Let
$\{e_1,\ldots,e_{2n}\}$ denote the standard symplectic basis of
$E:=\C^{2n}$ and let $F_i$ be the subspace spanned by the first $i$
vectors of this basis, so that $F_{n-i}^\perp = F_{n+i}$ for $0\leq i
\leq n$. Let $B$ denote the stabilizer of the flag $F_\bull$ in the
symplectic group $\Sp_{2n}=\Sp_{2n}(\C)$, and let $T$ be the
associated maximal torus in the Borel subgroup $B$. The symplectic
flag manifold given by $\M_n:=\Sp_{2n}/B$ parametrizes complete flags
$E_\bull$ in $E$ with $E_{n-i}^\perp = E_{n+i}$ for $0\leq i \leq n$.
The $T$-equivariant cohomology ring $\HH^*_{T}(\M_n)$ is defined as
the cohomology ring of the Borel mixing space $ET\times^{T}\M_n$. The
ring $\HH^*_{T}(\M_n)$ is a $\Z[Y_n]$-algebra, where $y_i$ is identified
with the equivariant Chern class $-c_1^T(F_{n+1-i}/F_{n-i})$, for
$1\leq i \leq n$.

The Schubert varieties in $\M_n$ are the closures
of the $B$-orbits, and are indexed by the elements of
$W_n$. Concretely, any $w\in W_n$ corresponds to a Schubert variety
$X_w=X_w(F_\bull)$ of codimension $\ell(w)$, defined by
\[
   X_w := \{ E_\bull \in \M_n \mid \dim(E_r \cap
   F_s) \geq d_w(r,s) \ \, \mathrm{for} \ 1\leq r \leq n, \, 1\leq s \leq 2n \},
\]
where $d_w(r,s)$ is the rank function specified as follows. 
Consider the group monomorphism $\zeta:W_n\hra S_{2n}$ with image
\[
\zeta(W_n)=\{\,\om\in S_{2n} \ |\ \om_i+\om_{2n+1-i} = 2n+1,
 \ \ \text{for all}  \ i\,\},
\]
and determined by setting, for each $w=(w_1,\ldots,w_n)\in W_n$ and
$1\leq i \leq n$,
\[
\zeta(w)_i :=\left\{ \begin{array}{cl}
             n+1-w_{n+1-i} & \mathrm{ if } \ w_{n+1-i} \ \mathrm{is} \
             \mathrm{unbarred}, \\
             n+\ov{w}_{n+1-i} & \mathrm{otherwise}.
             \end{array} \right.
\]
Then $d_w(r,s)$ equals the number of $i\leq r$ such that $\zeta(w)_i >
2n-s$.  Since $X_w$ is stable under the action of $T$, we obtain an
{\em equivariant Schubert class} $[X_w]^T:=[ET\times^{T}X_w]$ in
$\HH^*_{T}(\M_n)$.

Following \cite{IMN1}, there is a surjective homomorphism of graded 
$\Z[Y_n]$-algebras
\[
\pi_n:\Gamma[X_n,Y_n] \to \HH^*_{T}(\M_n)
\]
such that 
\[
\pi_n(\CS^{(n)}_w)= \begin{cases}
[X_w]^T & \text{if $w\in W_n$}, \\
0 & \text{if $w\in W^{(n)}\smallsetminus W_n$}.
\end{cases}
\]
We let $E_i$ denote the $i$-th tautological vector vector bundle over
$\M_n$, for $0\leq i \leq 2n$.  The geometrization map $\pi_n$ is
defined by the equations
\[
\pi_n(x_i) = c_1^T(E_{n+1-i}/E_{n-i}) \ \ \text{and} \ \ 
\pi_n(c_p) = c_p^T(E-E_n-F_n)
\]
for $1\leq i \leq n$ and $p\geq 1$. Here $c_p^T(E-E_n-F_n)$ denotes the 
degree $p$ component of the total Chern class 
$c^T(E-E_n-F_n):=c^T(E)c^T(E_n)^{-1}c^T(F_n)^{-1}$.

\subsection{The kernel of the map $\pi_n$}
For any integer $j\geq 0$ and sequence of variables 
$Z=(z_1,z_2,\ldots)$, define the
elementary and complete symmetric functions 
$e_j(Z)$ and $h_j(Z)$ by the generating series
\[
\prod_{i=1}^{\infty}(1+z_it) = \sum_{j=0}^{\infty}
e_j(Z)t^j \ \ \ \text{and} \ \ \ 
\prod_{i=1}^{\infty}(1-z_it)^{-1} = \sum_{j=0}^{\infty}
h_j(Z)t^j,
\]
respectively. 
If $r\geq 1$ then we let $e^r_j(Z):=e_j(z_1,\ldots,z_r)$ 
and $h^r_j(Z):=h_j(z_1,\ldots,z_r)$ denote the polynomials
obtained from $e_j(Z)$ and $h_j(Z)$ by setting $z_j=0$ for all $j>r$. Let 
$e^0_j(Z)=h^0_j(Z)=\delta_{0j}$, where $\delta_{0j}$ denotes the
Kronecker delta, and for $r<0$, define $h^r_j(Z):=e^{-r}_j(Z)$
and $e^r_j(Z):=h^{-r}_j(Z)$.

For any $k,k'\in \Z$, define the polynomial
${}^kc^{k'}_p={}^kc^{k'}_p(X,Y)$ by
\[
{}^kc^{k'}_p:=\sum_{i=0}^p\sum_{j=0}^p c_{p-j-i}h^{-k}_i(X)h_j^{k'}(-Y).
\]

\begin{defn}
Let $$\wh{\Gamma}^{(n)}:= \Z[{}^nc^n_1, {}^nc^n_2,\ldots]$$ and let
$\wh{\IGam}^{(n)}$ be the ideal of $\Gamma[X_n,Y_n]$ generated by
the homogeneous elements in $\wh{\Gamma}^{(n)}$ of positive degree.
\end{defn}

For any $p\in \Z$, define $\wh{e}_p\in \Z[X_n,Y_n]$ by
\[
\wh{e}_p=\wh{e}_p(X_n/Y_n):=\sum_{i+j=p}e_i(X_n)h_j(-Y_n).
\]
We then have the generating function equation
\begin{equation}
\label{genfuneq}
\sum_{p=0}^\infty {}^nc^n_p\,t^p =
\left(\sum_{p=0}^\infty c_pt^p\right)\left(\sum_{j=0}^n\wh{e}_jt^j\right)=
\left(\sum_{p=0}^\infty c_pt^p\right)\prod_{j=1}^n\frac{1+x_jt}{1+y_jt}.
\end{equation}

\begin{lemma}
\label{IGlem}
We have $\wh{\IGam}^{(n)}\subset \Ker\pi_n$.
\end{lemma}
\begin{proof}
It suffices to show that ${}^nc^n_p \in \Ker\pi_n$ for each $p\geq
1$. We give two proofs of this result. A straightforward calculation
using Chern roots shows that
\[
\pi_n({}^kc^{k'}_p) = c^T_p(E-E_{n-k}-F_{n+k'})
\]
for all $p,k,k'\in \Z$. Since $E=F_{2n}$, we deduce the lemma 
from this and the properties of Chern classes.

Our second proof proceeds as follows. There is a canonical isomorphism 
of $\Z[Y_n]$-algebras
\[
\HH_{T}^*(\M_n) \cong \Z[A_n,Y_n]/K_n,
\]
where $A_n:=(a_1,\ldots,a_n)$
and $K_n$ is the ideal of $\Z[A_n,Y_n]$
generated by the differences $e_i(A_n^2)-e_i(Y_n^2)$ for $1 \leq i \leq n$
(see for example \cite[\S 3]{F2}). The geometrization 
map $\pi_n$ satisfies $\pi_n(x_j)=-a_j$ for $1\leq j \leq n$, while 
\[
\pi_n(c_p):= \sum_{i+j=p}e_i(A_n)h_j(Y_n), \ \ \ p\geq 0.
\]
A straightforward calculation using (\ref{genfuneq}) gives
\[
\pi_n\left(\sum_{p=0}^\infty {}^nc^n_p\,t^p\right) = 
\prod_{j=1}^n\frac{1-a^2_jt^2}{1-y_j^2t^2}.
\]
On the other hand, we have
\begin{gather*}
\prod_{j=1}^n\frac{1-a^2_jt^2}{1-y_j^2t^2} = 
1 + \left(\prod_{j=1}^n(1-a^2_jt^2) - \prod_{j=1}^n(1-y^2_jt^2)\right)
\cdot\sum_{p=0}^\infty h_p(Y_n^2)t^{2p} \\
= 1+\left(\sum_{r=0}^n(-1)^r(e_r(A_n^2) - e_r(Y_n^2))t^{2r}\right) 
\cdot\sum_{p=0}^\infty h_p(Y_n^2)t^{2p}.
\end{gather*}
The result follows immediately.
\end{proof}

For any three integer vectors $\al,\be,\rho\in \Z^\ell$, which we view
as integer sequences with finite support, define
${}^{\rho}c^\be_\al:={}^{\rho_1}c^{\be_1}_{\al_1}\,{}^{\rho_2}c^{\be_2}_{\al_2}\cdots$.
Recall that for each $i<j$, the operator $R_{ij}$ acts on
integer sequences $\al=(\al_1,\al_2,\ldots)$ by 
\[
R_{ij}(\alpha) := (\alpha_1,\ldots,\alpha_i+1,\ldots,\alpha_j-1,
\ldots).
\]
A {\em raising operator} $R$ is any monomial in these $R_{ij}$'s.
Given any raising operator $R=\prod_{i<j}R_{ij}^{n_{ij}}$, let
$R\, {}^{\rho}c^\be_{\al} := {}^{\rho}c^\be_{R\al}$.  Finally, define
the {\em multi-Schur Pfaffian} ${}^\rho Q^\be_\al$ by
\begin{equation}
\label{Qdef}
{}^\rho Q^\be_\al:= \RR\, {}^\rho c^\be_\al,
\end{equation}
where the raising operator expression $\RR$ is given by
\[
\RR:=\prod_{i<j}\frac{1-R_{ij}}{1+R_{ij}}\,.
\]
The name `multi-Schur Pfaffian' is justified because
${}^\rho Q^\be_\al$ is equal to the Pfaffian of the $r\times r$
skew-symmetric matrix with
\[
\left\{{}^{\rho_i,\rho_j}Q^{\be_i,\be_j}_{\al_i,\al_j}\right\}_{1\leq i<j\leq r}=
\left\{\frac{1-R_{12}}{1+R_{12}}\,
{}^{\rho_i,\rho_j}c^{\be_i,\be_j}_{\al_i,\al_j}\right\}_{1\leq i<j\leq r}
\]
above the main diagonal, following Kazarian \cite{K};
here $r=2\lfloor\ell/2\rfloor$.
We adopt the convention that when some superscript(s) are omitted, the
corresponding indices are equal to zero. Thus ${}^kc_p:={}^kc^0_p$,
$c^{k'}_p:={}^0c^{k'}_p$,
${}^\rho c_{\al}:=\prod_i {}^{\rho_i}c^0_{\al_i}$,
${}^\rho Q_\al:=\RR\, {}^\rho c_{\al}$, $Q_\al:=\RR\, c_\al$, etc.

If $\la=(\la_1>\la_2>\cdots > \la_\ell)$ is a strict partition of
length $\ell$, let $w_\la$ be the corresponding increasing Weyl group
element, so that the negative components of $w_\la$ are exactly
$(-\la_1,\ldots,-\la_\ell)$.

\begin{lemma}
\label{Qprop}
If $\la$ is a strict partition with $\la_1>n$, then 
$\CS^{(n)}_{w_\la}(X_n,Y_n)\in \wh{\IGam}^{(n)}$.
\end{lemma}
\begin{proof}

For $p\geq 0$,  recall that $c^{-n}_p:={}^0c^{-n}_p\in \Gamma[Y_n]$, so that 
we have
\[
\sum_{p=0}^\infty c_p^{-n}\,t^p =
\left(\sum_{p=0}^\infty c_pt^p\right)\prod_{j=1}^n(1-y_jt).
\]
One has the generating function equation
\begin{equation}
\label{hqeq}
\left(\sum_{p=0}^\infty {}^nc^n_p\,t^p\right)\left(\sum_{p=0}^\infty 
c^{-n}_p(-t)^p\right) = \sum_{j=0}^ne_j(X_n)t^j.
\end{equation}
It follows from (\ref{hqeq}) that
\[
{}^nc^n_p-{}^nc^n_{p-1}c^{-n}_1+\cdots + (-1)^p (c^{-n}_p) = e_p(X_n)
\]
for each $p\geq 0$. We deduce that $c^{-n}_p\in \wh{\IGam}^{(n)}$ 
when $p\geq n+1$.

According to \cite[Thm.\ 6.6]{IMN1}, we have 
\begin{equation}
\label{rof}
\CS_{w_\la}(X,Y) = Q_\la^{\be(\la)} = R^\infty\, c_\la^{\be(\la)}
\end{equation}
in $\Gamma[X,Y]$, where $\be(\la)$ is equal to the integer vector
$(1-\la_1,\ldots,1-\la_\ell)$.  It follows that
\begin{equation}
\label{Qleq}
\CS^{(n)}_{w_\la}(X_n,Y_n) = \ov{Q}^{\be(\la)}_{\la}
\end{equation}
where $\ov{Q}^{\be(\la)}_{\la}$ is obtained from $Q_\la^{\be(\la)}$
by setting $y_j=0$ for all $j>n$. The conclusion of the lemma now
follows immediately by expanding the raising operator formula
(\ref{rof}) for the double Schur $Q$-polynomial
$\ov{Q}^{\be(\la)}_{\la}$ and noting that in each monomial of the
result, the first factor is equal to $c^{-n}_p$ for some $p>n$.
\end{proof}

\begin{lemma}
\label{CSlem}
For any $w\in W_{\infty}\ssm W_n$, we have $\CS^{(n)}_w\in \wh{\IGam}^{(n)}$.
\end{lemma}
\begin{proof}
For any positive integers $i<j$ we define reflections $t_{ij}\in S_\infty$
and $\ov{t}_{ij},\ov{t}_{ii} \in W_\infty$ by their right actions
\begin{align*}
(\ldots,w_i,\ldots,w_j,\ldots)\,t_{ij} &= 
(\ldots,w_j,\ldots,w_i,\ldots), \\
(\ldots,w_i,\ldots,w_j,\ldots)\,\ov{t}_{ij} &= 
(\ldots,\ov{w}_j,\ldots,\ov{w}_i,\ldots), \ \ \mathrm{and} \\
(\ldots,w_i,\ldots)\,\ov{t}_{ii} &= 
(\ldots,\ov{w}_i,\ldots),
\end{align*} 
and let $\ov{t}_{ji} := \ov{t}_{ij}$.

Let $w$ be an element of $W_\infty$. According to \cite[Lemma 2]{B},
if $i\leq j$, then $\ell(w\ov{t}_{ij})= \ell(w)+1$ if and only if (i)
$-w_i<w_j$, (ii) in case $i<j$, either $w_i<0$ or $w_j<0$, and (iii)
there is no $p<i$ such that $-w_j<w_p<w_i$, and no $p<j$ such that
$-w_i < w_p < w_j$.

The group $W_\infty$ acts on the polynomial ring $\Z[y_1,y_2,\ldots]$
in the usual way, with $s_i$ for $i\geq 1$ interchanging $y_i$ and
$y_{i+1}$ and leaving all the remaining variables fixed, and $s_0$
mapping $y_1$ to $-y_1$ and fixing the $y_j$ with $j\geq 2$.  Let
$w\in W_\infty$ be non-increasing, let $r$ be the last positive
descent of $w$, let $s:=\max(i>r\ |\ w_i<w_r)$, and let $v:= wt_{rs}$.
Following \cite[Prop.\ 6.12]{IMN1}, the double Schubert polynomials
$\CS_u=\CS_u(X,Y)$ obey the {\em transition equations}
\begin{equation}
\label{Ctrans}
\CS_w = (x_r-v(y_r))\CS_v + \sum_{{1\leq i < r} \atop
{\ell(vt_{ir}) = \ell(w)}} \CS_{vt_{ir}} + 
\sum_{{i\geq 1} \atop {\ell(v\ov{t}_{ir}) = 
\ell(w)}} \CS_{v\ov{t}_{ir}}
\end{equation}
in $\Gamma[X,Y]$. The recursion (\ref{Ctrans}) terminates in a
$\Z[X,Y]$-linear combination of elements $\CS_{w_\nu}(X,Y)$ for
strict partitions $\nu$.

For any $w\in W_\infty$, let $\mu(w)$ denote the strict partition
whose parts are the elements of the set $\{|w_i|\ :\ w_i<0\}$. Clearly
we have $\mu(w)=\mu(wu)$ for any $u\in S_\infty$. In equation
(\ref{Ctrans}), we therefore have $\mu(v)=\mu(vt_{ir}) =
\mu(w)$. Moreover, condition (i) above shows that the parts of
$\mu(v\ov{t}_{ir})$ are greater than or equal to the parts of
$\mu(w)$. In particular, if $\mu(w)_1>n$, then
$\mu(v\ov{t}_{ir})_1>n$.

Assume first that $w\in W_{n+1}\ssm W_n$. If $w_i=-n-1$ for some
$i\leq n+1$, we use the transition recursion (\ref{Ctrans}) to write
$\CS^{(n)}_w$ as a $\Z[X_n,Y_n]$-linear combination of elements
$\CS^{(n)}_{w_\nu}$ for strict partitions $\nu$ with
$\nu_1>n$. Lemma \ref{Qprop} now implies that $\CS^{(n)}_w\in
\wh{\IGam}^{(n)}$.

Next, we consider the case when $w_i=n+1$ for some $i\leq n$. Let 
$$\{v_2,\ldots, v_n\} := \{w_1,\ldots,\wh{w_i},\ldots, w_n\}$$ with 
$v_2>\cdots > v_n$, and define
\[
u:=(n+1,v_2,\ldots,v_n,w_{n+1})\in W_{n+1}
\]
and 
\[
\ov{u}:=us_0=(\ov{n+1},v_2,\ldots,v_n,w_{n+1}).
\]
We have $\CS^{(n)}_{\ov{u}}\in \wh{\IGam}^{(n)}$ from the previous case,
and, using (\ref{ddCeqinit}), that $\partial_0(\CS^{(n)}_{\ov{u}})= \CS^{(n)}_u$.

For any integer $i\in [0,n-1]$, it is easy to check that
$s_i({}^nc_p^n)={}^nc_p^n$, and therefore that
$\partial_i({}^nc_p^n)=0$.  It follows that
$\partial_i(\wh{\IGam}^{(n)})\subset \wh{\IGam}^{(n)}$ for all indices
$i\in [0,n-1]$. Since $\CS^{(n)}_{\ov{u}}\in \wh{\IGam}^{(n)}$, we deduce
that $\CS^{(n)}_u\in \wh{\IGam}^{(n)}$. There exists a permutation
$\sigma\in S_n$ such that $u=w\sigma$ and
$\ell(\sigma)=\ell(u)-\ell(w)$.  Using (\ref{Cstar}), we have
$\CS^{(n)}_w = \partial_{\sigma}(\CS^{(n)}_u)$, and hence conclude that
$\CS^{(n)}_w$ lies in $\wh{\IGam}^{(n)}$.

Finally assume $w\notin W_{n+1}$ and let $m$ be minimal such that 
$w\in W_m$. Then $w\in W_m\ssm W_{m-1}$, so the above argument 
applies with $m-1$ in place of $n$. The result now follows by 
setting $x_j=y_j=0$ for all $j>n$.
\end{proof}

\begin{thm}
\label{Cthm}
Let $J_n:=\bigoplus_{w\in W^{(n)}\ssm W_n}\Z[Y_n]\CS^{(n)}_w$. Then we have 
\begin{equation}
\label{Jeq}
\wh{\IGam}^{(n)} = J_n=\sum_{w\in W_{\infty}\ssm W_n}
\Z[Y_n]\CS^{(n)}_w= \Ker \pi_n.
\end{equation}
We have a canonical isomorphism of $\Z[Y_n]$-algebras
\[
\HH_{T}^*(\Sp_{2n}/B) \cong \Gamma[X_n,Y_n]/\wh{\IGam}^{(n)}.
\]
\end{thm}
\begin{proof}
Lemmas \ref{IGlem} and \ref{CSlem} imply that 
\begin{equation}
\label{inclusions}
J_n\subset \sum_{w\in W_{\infty}\ssm W_n} \Z[Y_n]\CS^{(n)}_w \subset \wh{\IGam}^{(n)}
\subset \Ker \pi_n.
\end{equation}
We claim that $\Ker\pi_n\subset J_n$. Indeed, if $f\in \Ker\pi_n$ then
by Corollary \ref{basiscor2} we have a unique expression
\begin{equation}
\label{ffeq}
f = \sum_{w\in W^{(n)}} f_w \CS^{(n)}_w
\end{equation}
for some coefficients $f_w\in \Z[Y_n]$. Applying the map $\pi_n$ to (\ref{ffeq})
and using (\ref{inclusions}) gives
\[
\sum_{w\in W_n}f_w [X_w]^T = 0.
\]
Since the equivariant Schubert classes $[X_w]^T$ are a $\Z[Y_n]$-basis of 
$\HH^*_{T}(\Sp_{2n}/B)$, we deduce that $f_w=0$ for all $w\in W_n$. It 
follows that $f\in J_n$.
\end{proof}

\begin{remark}
  (a) It is easy to show that ${}^nc^n_p$ lies in
  $\sum_{w\in W_{\infty}\ssm W_n} \Z[Y_n]\CS^{(n)}_w$ for all $n,p\geq
  1$.
  This follows from the fact that
  ${}^nc^n_p = {}^{n+p-1}c^{(n+p-1)+1-p}_p(X_n,Y_n)$ is equal to the
  (restricted) double Schubert polynomial $\CS^{(n)}_{w_{(p)}}$, where
\[
w_{(p)}:=s_ns_{n+1}\cdots s_{n+p-1}. 
\]
In fact, $\CS_{w_{(p)}}(X,Y)$ is equal to the double theta polynomial
${}^{n+p-1}\Ti_p(X,Y)$ of level $n+p-1$, for every $p\geq 1$ (the
definition of ${}^{n+p-1}\Ti_p(X,Y)$ is recalled in (\ref{Cdef})).

\medskip
\noin
(b) The equality $\sum_{w\in W_{\infty}\ssm W_n}\Z[Y_n]\CS^{(n)}_w= \Ker \pi_n$ in 
(\ref{Jeq}) was proved earlier in \cite[Prop.\ 7.7]{IMN1} using 
different methods.
\end{remark}

\medskip

For any elements $f,g\in \Gamma[X_n,Y_n]$, we define the congruence
$f\equiv g$ to mean $f-g\in \wh{\IGam}^{(n)}$. We claim that any
element of $\Gamma[X_n,Y_n]$ is equivalent under $\equiv$ to a
polynomial in $\Z[X_n,Y_n]$. Indeed, we have that
\begin{equation}
\label{tqeh}
\left(\sum_{p=0}^\infty {}^nc^n_p\,t^p\right)\left(\sum_{p=0}^\infty c_p(-t)^p\right) =
\sum_{j=0}^\infty\wh{e}_jt^j.
\end{equation}
It follows from (\ref{tqeh}) that
\begin{equation}
\label{ehtoq}
{}^nc^n_p-{}^nc^n_{p-1}c_1+\cdots + (-1)^r c_p=\wh{e}_p
\end{equation}
for each $p\geq 0$. The relation (\ref{ehtoq}) implies that $c_p
\equiv (-1)^p\wh{e}_p(X_n/Y_n)$, for all $p\geq 0$, proving the claim.

We deduce that $c_\al \equiv (-1)^{|\al|}e_\al(X_n/Y_n)$ for each
integer sequence $\al$, and that
\[
Q_\la = Q_\la(c) \equiv (-1)^{|\la|}\wt{Q}_\la(X_n/Y_n)
\]
for any partition $\la$. Here $\wt{Q}_\la(X_n/Y_n)$ denotes a
supersymmetric $\wt{Q}$-polynomial, namely
\[
\wt{Q}_\la(X_n/Y_n) := R^\infty \,\wh{e}_\la(X_n/Y_n).
\]
The reader can compare this with the remarks in \cite[\S 7.3]{T5}.

\subsection{Partial symplectic flag manifolds}
\label{psfms}

Following \cite{Bo, KK}, there is a standard way to generalize the presentation 
in Theorem \ref{Cthm} to partial flag manifolds $\Sp_{2n}/P$, where $P$
is a parabolic subgroup of $\Sp_{2n}$. The parabolic subgroups $P$ containing 
$B$ correspond to sequences $\fraka \ :\ a_1<\cdots < a_p$ of nonnegative
integers with $a_p<n$. The manifold $\Sp_{2n}/P$ parametrizes partial 
flags of subspaces 
\[
0 \subset E_1 \subset \cdots \subset E_p \subset E=\C^{2n}
\]
with $\dim(E_j) = n-a_{p+1-j}$ for each $j\in [1,p]$ and $E_p$ isotropic. 

A sequence $\fraka$ as above also parametrizes the parabolic
subgroup $W_P$ of $W_n$, which is generated by the simple reflections
$s_i$ for $i\notin\{a_1,\ldots, a_p\}$. Let $\Gamma[X_n,Y_n]^{W_P}$ be
the subring of elements in $\Gamma[X_n,Y_n]$ which are fixed by 
the action of $W_P$, that is, 
\[
\Gamma[X_n,Y_n]^{W_P} = \{ f\in \Gamma[X_n,Y_n]\ |\ 
s_i(f)=f, \ \forall\, i \notin\{a_1,\ldots, a_p\}, \ i<n\}.
\]
Since the action of $W_n$ on $\Gamma[X_n,Y_n]$ is $\Z[Y_n]$-linear, we
see that $\Gamma[X_n,Y_n]^{W_P}$ is a $\Z[Y_n]$-subalgebra of
$\Gamma[X_n,Y_n]$. Let $W^P\subset W^{(n)}$ denote the set
\[
W^P := \{w\in W^{(n)}\ |\ \ell(ws_i) = \ell(w)+1,\  \forall\, i \notin
\{a_1,\ldots, a_p\}, \ i<n\}.
\]

\begin{prop}
\label{wpprop}
We have 
\begin{equation}
\label{topreq}
\Gamma[X_n,Y_n]^{W_P} = \bigoplus_{w\in W^P} \Z[Y_n]\CS^{(n)}_w.
\end{equation}
\end{prop}
\begin{proof}
If $f$ is any element in $\Gamma[X_n,Y_n]^{W_P}$, Corollary
\ref{basiscor2} implies that we have an expansion $f = \sum_{w\in
W^{(n)}} d_w\CS^{(n)}_w$ for some coefficients $d_w$ in $\Z[Y_n]$.  If
$u\notin W^P$, there is an index $i<n$ with $i\notin \{a_1,\ldots,
a_p\}$ and $\ell(us_i)=\ell(u)-1$. We have $\partial_if=0$, and
on the other hand, using (\ref{ddCeqinit}), we see that
\[
 \partial_if = \sum_u d_u\CS^{(n)}_{us_i}
\]
summed over all $u$ such that $\ell(us_i)=\ell(u)-1$. It follows 
that $d_u=0$ for all such $u$, and thus that $\Gamma[X_n,Y_n]^{W_P}$ is
contained in the sum on the right hand side of (\ref{topreq}).

For the reverse inclusion, is suffices to show that 
$\CS^{(n)}_w\in \Gamma[X_n,Y_n]^{W_P}$ for all $w\in W^P$. The definition of 
$W^P$ implies that we have $\partial_i\CS^{(n)}_w=0$, or equivalently
$s_i\CS^{(n)}_w=\CS^{(n)}_w$,  for all $i<n$ with 
$i \notin \{a_1,\ldots, a_p\}$. The result follows.
\end{proof}

\begin{cor}
\label{GPcor}
There is a canonical isomorphism of $\Z[Y_n]$-algebras
\[
\HH_{T}^*(\Sp_{2n}/P) \cong \Gamma[X_n,Y_n]^{W_P}/\wh{\IGam}^{(n)}_P
\]
where $\wh{\IGam}^{(n)}_P$ is the ideal of $\Gamma[X_n,Y_n]^{W_P}$
generated by the homogeneous elements in $\wh{\Gamma}^{(n)}$ of
positive degree.
\end{cor}
\begin{proof}
It is well known that the canonical projection map $h:G/B \to G/P$
induces an injection $h^*: \HH_{T}^*(\Sp_{2n}/P) \hookrightarrow
\HH_{T}^*(\Sp_{2n}/B)$ on equivariant cohomology rings, with the image
of $h^*$ equal to the $W_P$-invariants in $\HH_{T}^*(\Sp_{2n}/B)$
(see for example \cite[Cor.\ (3.20)]{KK}). In fact, since the 
$\CS^{(n)}_w$ for $w\in W^P\cap W_n$ represent the
equivariant Schubert classes coming from $\HH_{T}^*(\Sp_{2n}/P)$, 
we deduce from Proposition \ref{wpprop} that the restriction of 
the geometrization map $\pi_n:\Gamma[X_n,Y_n] \to \HH^*_{T}(\Sp_{2n}/B)$
to the $W_P$-invariants induces a surjection 
\[
\Gamma[X_n,Y_n]^{W_P}\to \HH_{T}^*(\Sp_{2n}/P). 
\]
The result follows easily from this and Theorem \ref{Cthm}.
\end{proof}

\section{Divided differences and double theta polynomials}
\label{ddstheta}

\subsection{Preliminaries}
For every $i\geq 0$, the divided difference operator
$\partial_i=\partial_i^x$ on $\Gamma[X,Y]$ satisfies the Leibnitz rule
\begin{equation}
\label{LeibR}
\partial_i(fg) = (\partial_if)g+(s_if)\partial_ig.
\end{equation}

Observe that $\omega({}^kc^r_p) = {}^{-r}c^{-k}_p$, for all $k,r,p\in
\Z$.  By applying this, it is easy to prove the following dual versions of
\cite[Lemmas 5.4 and 8.2]{IM}.

\begin{lemma}
\label{ddylem}
Suppose that $k,p,r\in \Z$.
For all $i\geq 0$, we have
\[
\partial_i ({}^kc_p^r)= 
\begin{cases}
{}^{k-1}c_{p-1}^r & \text{if $k=\pm i$}, \\
0 & \text{otherwise}.
\end{cases}
\]
\end{lemma}

\begin{lemma}
\label{imlem1dual}
Suppose that $k\geq 0$ and $r\geq 1$. Then we have
\[
{}^kc_p^{-r} = {}^{k+1}c_p^{-r+1}-(x_{k+1}+y_r)\, {}^kc_{p-1}^{-r+1}.
\]
\end{lemma}

We also require the following lemma.

\begin{lemma}[\cite{IM}, Prop.\ 5.4]
\label{imlem2}
Suppose that $k,r\geq 0$ and $p> k+r$. Then we have
\[
{}^{(k_1,\ldots,k,k,\ldots,k_\ell)}
Q_{(p_1,\ldots,p,p,\ldots,p_\ell)}^{(r_1,\ldots,-r,-r,\ldots,r_\ell)}= 0.
\]
\end{lemma}

\subsection{The shape of a signed permutation}
\label{ssp}
We proceed to define certain statistics of an element of $W_\infty$.

\begin{defn}
\label{codedef}
Let $w\in W_\infty$ be a signed permutation. The strict partition
$\mu=\mu(w)$ is the one whose parts are the absolute values of the
negative entries of $w$, arranged in decreasing order. The {\em
  A-code} of $w$ is the sequence $\gamma=\gamma(w)$ with
$\gamma_i:=\#\{j>i\ |\ w_j<w_i\}$. We define a partition
$\delta=\delta(w)$ whose parts are the non-zero entries $\gamma_i$
arranged in weakly decreasing order, and let $\nu(w):=\delta(w)'$ be
the conjugate of $\delta$. Finally, the {\em shape} of $w$ is the
partition $\la(w):=\mu(w)+\nu(w)$.
\end{defn}

It is easy to see that $w$ is uniquely determined by $\mu(w)$ and
$\gamma(w)$, and that $|\la(w)|=\ell(w)$. The shape $\la(w)$ of an element 
$w\in W_\infty$ is a natural generalization of the shape of a permutation,
as defined in \cite[Chp.\ 1]{M2}.

\begin{example}
(a) For the signed permutation $w := (\ov{3}, 2, \ov{7}, \ov{1}, 5, 4,
  \ov{6})$ in $W_7$, we obtain $\mu = (7,6,3,1)$, $\gamma=(2,3, 0, 1,
2, 1, 0)$, $\delta = (3,2,2,1,1)$, $\nu= (5, 3, 1)$, and $\la = (12,
9, 4, 1)$.

\smallskip \noin (b) An element $w\in W_\infty$ is $n$-Grassmannian if
$\ell(ws_i)>\ell(w)$ for all $i\neq n$, while a partition $\la$ is
called $n$-strict if all its parts $\la_i$ greater than $n$ are
distinct.  Following \cite[\S 6.1]{BKT1}, these two objects are in
one-to-one correspondence with each other. If $w$ is an
$n$-Grassmannian element of $W_\infty$, then $\la(w)$ is the
$n$-strict partition associated to $w$, in the sense of op.\ cit.
\end{example}

\begin{lemma}
\label{Mcdlem}
If $i\geq 1$, $w\in W_\infty$, and $\gamma=\gamma(w)$, then 
\[
\gamma_i>\gamma_{i+1} \Leftrightarrow w_i>w_{i+1} \Leftrightarrow
\ell(ws_i)=\ell(w) -1.
\]
If any of the above conditions hold, then 
\[
\gamma(ws_i) = (\gamma_1,\ldots, \gamma_{i-1},
\gamma_{i+1},\gamma_i-1,\gamma_{i+2},\gamma_{i+3},\ldots).
\]
\end{lemma}
\begin{proof}
These follow immediately from \cite[(1.23) and (1.24)]{M2}.
\end{proof}

Let $\beta(w)$ be the sequence defined by
$\be(w)_i=\min(1-\mu(w)_i,0)$ for each $i\geq 1$. For each $n\geq 1$,
let $w_0^{(n)}:=(\ov{1},\ldots,\ov{n})$ denote the longest element in
$W_n$.

\begin{prop}
\label{PfCprop}
Suppose that $m>n\geq 0$ and $w\in W_m$ is an $n$-Grassmannian element. 
Set $\wh{w}:=ww_0^{(n)}$. Then we have 
\[
\CS_{\wh{w}}(X,Y) = {}^{\nu(\wh{w})}Q_{\la(\wh{w})}^{\beta(\wh{w})}
\]
in the ring $\Gamma[X_n,Y_{m-1}]$. In particular, if $w\in S_m$, then
we have
\[
\CS_{\wh{w}}(X,Y) = {}^{\delta_{n-1}}Q_{\delta_n+\delta_{n-1}+\la(w)'}^{(1-w_n,\ldots,1-w_1)}.
\]
\end{prop}
\begin{proof}
We first consider the case when $w\in S_m$. We have
\[
w = (a_1,\ldots,a_n,d_1,\ldots, d_r)
\]
where $r=m-n$, $0<a_1<\cdots<a_n$ and $0<d_1<\cdots<d_r$. If $\la:=\la(w)$ then
\[
\la_j=n+j-d_j = m-d_j-(r-j) \ \ \ \text{for $1\leq j \leq r$}.
\]
 
Let $w_0^{(m)}:=(\ov{1},\ldots,\ov{m})$ be the longest element in $W_m$.
Then we have $w_0^{(m)}=\wh{w}v_1\cdots v_r$,
where $\ell(w_0^{(m)})=\ell(\wh{w})+\sum_{j=1}^r\ell(v_j)$ and
\[
v_j=s_{n+j-1}\cdots s_1s_0s_1\cdots s_{d_j-1}, \ \ 1\leq j \leq r.
\]
One knows from \cite[Thm.\ 1.2]{IMN1} and \cite[Prop.\ 3.2]{T7}
that the equation 
\[
\CS_{w^{(m)}_0}(X,Y) 
= {}^{\delta_{m-1}}Q_{\delta_m+\delta_{m-1}}^{-\delta_{m-1}}
\]
holds in $\Gamma[X,Y]$. It follows from this and (\ref{Cstar}) that
\begin{equation}
\label{Cwn}
\CS_{\wh{w}} = \partial_{v_1}\cdots \partial_{v_r}\left(\CS_{w_0^{(m)}}\right) = 
\partial_{v_1}\cdots \partial_{v_r} \left(
{}^{\delta_{m-1}}Q_{\delta_m+\delta_{m-1}}^{-\delta_{m-1}}\right).
\end{equation}

Using Lemmas \ref{ddylem} and \ref{imlem1dual}, for any $p,q\in \Z$ with 
$p\geq 1$, we obtain
\begin{equation}
\label{pqeq}
\partial_p({}^pc^{-p}_q) = {}^{p-1}c^{-p}_{q-1} = {}^pc^{1-p}_{q-1} 
-(x_p+y_p)\,{}^{p-1}c^{1-p}_{q-2}.
\end{equation}
Let $\epsilon_j$ denote the $j$-th standard basis vector in $\Z^m$. The
Leibnitz rule and (\ref{pqeq}) imply that for any integer vector
$\al = (\al_1,\ldots,\al_m)$, we have 
\[
\partial_p\left({}^{\delta_{m-1}}c^{-\delta_{m-1}}_{\al} \right)= 
{}^{\delta_{m-1}}c^{-\delta_{m-1}+\epsilon_{m-p}}_{\al-\epsilon_{m-p}}-
(x_p+y_p)\,{}^{\delta_{m-1}-\epsilon_{m-p}}
c^{-\delta_{m-1}+\epsilon_{m-p}}_{\al - 2\epsilon_{m-p}}.
\]
We deduce from this and Lemma \ref{imlem2} that 
\begin{align*}
\partial_p
{}^{\delta_{m-1}}Q_{\delta_m+\delta_{m-1}}^{-\delta_{m-1}} &= 
{}^{\delta_{m-1}}Q^{-\delta_{m-1}+\epsilon_{m-p}}_{\delta_m+\delta_{m-1}-\epsilon_{m-p}} \\ 
&= {}^{(m-1,\ldots,1,0)}
Q_{(2m-1,\ldots,2p+3, 2p, 2p-1,2p-3,\ldots,1)}^{(1-m,\ldots, -1-p, 1-p,1-p,2-p,\ldots, -1,0)}.
\end{align*}
Iterating this calculation for $p=d_r-1,\ldots,1$ gives
\[
(\partial_1\cdots\partial_{d_r-1})\CS_{w_0^{(m)}} = 
{}^{(m-1,\ldots,1,0)}
Q_{(2m-1,\ldots,2d_r+1, 2d_r-2, 2d_r-4,\ldots,2,1)}^{(1-m,\ldots, -d_r, 2-d_r,3-d_r,\ldots, -1,0,0)}.
\]
Since $\partial_0({}^0c^0_1) = {}^{-1}c^0_0 = 1$, it follows that
\[
(\partial_0\partial_1\cdots\partial_{d_r-1})\CS_{w_0^{(m)}} =
  {}^{(m-1,\ldots,1)}Q_{(2m-1,\ldots,2d_r+1, 2d_r-2, 2d_r-4,\ldots,2)}^{(1-m,\ldots, -d_r, 2-d_r,3-d_r,\ldots, -1,0)}.
\]
Applying Lemma \ref{ddylem} alone $m-1$ times now gives
\[
\partial_{v_r}(\CS_{w_0^{(m)}}) = {}^{\delta_{m-2}}Q^{(1-m,\ldots,
  -d_r, 2-d_r,3-d_r,\ldots,
  -1,0)}_{(2m-2,\ldots,2d_r,2d_r-3,\ldots,1)} =
        {}^{\delta_{m-2}}Q^{(1-m,\ldots,\wh{1-d_r},\ldots,0)}_{\delta_{m-1}+\delta_{m-2}+1^{m-d_r}}.
\]
Finally, we use (\ref{Cwn}) and repeat the above calculation $r-1$ more
times to get
\[
\CS_{\wh{w}} = {}^{\delta_{n-1}}Q_{\delta_n+\delta_{n-1}+\xi}^{\rho}
\]
where 
\[
\rho = (1-m,\ldots,\wh{1-d_r},\ldots,\wh{1-d_1},\ldots,-1,0) = 
(1-w_n,\ldots,1-w_1)
\]
and 
\[
\xi = \sum_{j=1}^r1^{m-d_j-(r-j)} = \sum_{j=1}^r 1^{n+j-d_j} = 
\sum_{j=1}^r 1^{\la_j} = \la(w)'.
\]

Next consider the general case. Let $p>n$ and suppose that
\[
w = (a_1,\ldots,\wh{a}_{i_1},\ldots, \wh{a}_{i_{p-n}}, \ldots, a_p
,-a_{i_{p-n}},\ldots, -a_{i_1}, d_1,\ldots, d_r)
\]
where $r=m-p$, $0<a_1<\cdots<a_p$ and $0<d_1<\cdots<d_r$. If
\[
u:= (a_1,\ldots,a_p,d_1,\ldots, d_r)
\]
and $\wh{u}:=uw_0^{(p)}$, then 
\begin{equation}
\label{fact}
\wh{u}=\wh{w}v'_{p-n}\cdots v'_1
\end{equation}
where $v'_j = s_{p-j} \cdots s_{i_j-j+2}s_{i_j-j+1}$ for $1\leq j
\leq p-n$.  Now $\CS_{\wh{u}}$ is known by the previous case, and
\[
\CS_{\wh{w}} = \partial_{v'_{p-n}}\cdots\partial_{v'_1}(\CS_{\wh{u}}).
\]
The proof is now completed by induction, using Lemma \ref{Mcdlem}. The
key observation is the following: Suppose that
\[
u_0 = \wh{u} > u_1 > \cdots > u_d = \wh{w}
\]
is the sequence of coverings in the right weak Bruhat order
corresponding to the factorization (\ref{fact}), so that $u_{i+1} =
u_i s_{r_i}$ with $\ell(u_{i+1})=\ell(u_i)-1$ for each $i\in
[0,d-1]$. Then if $\gamma:=\gamma(u_i)$, we have $\gamma_{r_i+1} =
\gamma_{r_i}-1$. Therefore Lemma \ref{Mcdlem} implies that
$\gamma(u_{i+1})$ has two equal entries in positions $r_i$ and
$r_i+1$. Moreover, $\gamma(u_j)$ is a partition for all $j\in [0,d]$,
and hence $\nu(u_j)$ is the conjugate of $\gamma(u_j)$.
\end{proof}

\begin{remark}
  The work of Anderson and Fulton \cite{AF} associates a partition
  $\la$ to certain triples of $\ell$-tuples of integers which define a
  class of symplectic degeneracy loci. The shape $\la(w)$ of an
  element $w\in W_\infty$ in Definition \ref{codedef} (and its even
  orthogonal counterpart in Definition \ref{codedefD}) is consistent
  with op.\ cit. In particular, Propositions \ref{PfCprop} and
  \ref{PfDprop} follow from the more general formulas for double
  Schubert polynomials which are established in \cite{AF}. We give
  here an alternative proof, using \cite[Thm.\ 1.2]{IMN1} and the
  right divided difference operators.
\end{remark}

\subsection{Double theta polynomials and alternating sums}
Let $n\geq 0$ and $w \in W_\infty$ be an $n$-Grassmannian element. Let
$\la=\la(w)$ be the $n$-strict partition which corresponds to $w$, 
define a sequence $\beta(\la)=\{\beta_i(\la)\}_{i\geq 1}$ by 
\begin{equation}
\label{betaC}
\beta_i(\la):=\begin{cases} w_{n+i}+1 & \text{if $w_{n+i}<0$}, \\
w_{n+i} & \text{if $w_{n+i}>0$},
\end{cases}
\end{equation}
and a set of pairs $\cC(\la)$ by
\begin{equation}
\label{CC}
\cC(\la):=\{(i,j)\in \N\times\N \ |\ 1\leq i < j \ \ \text{and} \ \ 
w_{n+i}+w_{n+j}<0\}
\end{equation}
(this agrees with the set $\cC(\la)$ given in the introduction).
The {\em double theta polynomial} ${}^n\Ti_\la(X,Y)$ of
  \cite{TW, W} is defined by 
\begin{equation}
\label{Cdef}
{}^n\Ti_\la(X,Y) := 
\prod_{i<j} (1-R_{ij}) \prod_{(i,j)\in \cC(\la)}
(1+R_{ij})^{-1}\, ({}^nc)^{\beta(\la)}_\la.
\end{equation}
In the above formula, for any integer sequence
$\al=(\al_1,\al_2,\ldots)$, we let
$({}^nc)^{\be(\la)}_\al:=\prod_i{}^nc^{\be_i(\la)}_{\al_i}$, and the
raising operators $R_{ij}$ act by $R_{ij}
({}^nc)^{\be(\la)}_\al:=({}^nc)^{\be(\la)}_{R_{ij}\al}$. Note that
${}^n\Ti_\la(X,Y)$ lies in $\Gamma[X_n,Y]$ for any $n$-strict
partition $\la$. To be precise, the polynomial ${}^n\Ti_\la(X,Y)$ is
the image of the double theta polynomial $\Ti_\la(c \, |\, t)$ of
\cite{TW} (with $k=n$) in the ring $\Gamma[X_n,Y]$.

Let $\A:\Gamma[X_n,Y]\to \Gamma[X_n,Y]$ be the operator given by
\[
\A(f):=\sum_{w\in W_n}(-1)^{\ell(w)}w(f).
\]
Let $w_0=w_0^{(n)}$ denote the longest element in $W_n$ and set $\wh{w}:=ww_0$. 

\begin{thm}
\label{altCthm}
Let $\la$ be an $n$-strict partition and $w$ be the corresponding
$n$-Grassmannian element of $W_\infty$. Then we have
\begin{align}
\label{TtoQdoub}
{}^n\Ti_{\la}(X,Y) &= \partial_{w_0}\left({}^{\nu(\wh{w})}
Q^{\be(\wh{w})}_{\la(\wh{w})}\right) \\
\label{TtoQ2doub}
&=(-1)^{n(n+1)/2}\left. \A\left({}^{\nu(\wh{w})}
Q_{\la(\wh{w})}^{\be(\wh{w})}\right)\right\slash
\A\left(x^{\delta_n+\delta_{n-1}}\right).
\end{align}

\end{thm}
\begin{proof}
We deduce from (\ref{ddCeqinit}) that 
the double Schubert polynomial $\CS_w(X,Y)$ satisfies
\begin{equation}
\label{CtoC}
\CS_w(X,Y)=\partial_{w_0} \left(\CS_{\wh{w}}(X,Y)\right).
\end{equation}
The equality (\ref{TtoQdoub}) follows from (\ref{CtoC}), Proposition
\ref{PfCprop}, and the fact, proved in \cite[Thm.\ 1.2]{IM}, that 
$\CS_w(X,Y) = {}^n\Ti_{\la}(X,Y)$ in the ring $\Gamma[X_n,Y]$.

To establish the equality (\ref{TtoQ2doub}), recall from \cite[Lemma
  4]{D} and \cite[Prop.\ 5.5]{PR} that we have
\[
\partial_{w_0}(f) = (-1)^{n(n+1)/2}
\left(2^nx_1\cdots x_n\prod_{1\leq i < j\leq n}(x_i^2-x_j^2)
\right)^{-1} \cdot \A(f).
\]
On the other hand, it follows from \cite[Cor.\ 5.6(ii)]{PR} that
\[
\partial_{w_0}(x^{\delta_n+\delta_{n-1}}) = (-1)^{n(n+1)/2}
\] 
and hence that 
\[
\A(x^{\delta_n+\delta_{n-1}}) = 2^nx_1\cdots x_n\prod_{1\leq i < j\leq n}(x_i^2-x_j^2).
\]
The proof of (\ref{TtoQ2doub}) is completed by using these two
equations in (\ref{TtoQdoub}).
\end{proof}

\section{Single Schubert polynomials of type C}
\label{sspC}

In this section, we work with the single type C Schubert polynomials
$\CS_w(X)$. The entire section is inspired by \cite{LS1, LS2, M2} and
\cite{PR, LP1}.

\subsection{Theta polynomials as Weyl group invariants}
Let $\chi:\Gamma[X_n]\to \Z$ be the homomorphism defined by 
$\chi(c_p)=\chi(x_j)=0$ for all $p,j$. In other words, $\chi(f)$ is 
the constant term of $f$, for each polynomial $f\in \Gamma[X_n]$.

\begin{prop}
\label{etaprop}
For any $f\in \Gamma[X_n]$, we have 
$f=\sum_{w\in W^{(n)}}\chi(\partial_wf)\CS_w(X_n)$.
\end{prop}
\begin{proof}
By Proposition \ref{basisprop} and linearity, it is only necessary to
verify this when $f$ is a Schubert polynomial $\CS_v(X_n)$, $v\in
W^{(n)}$.  In this case, it follows from the properties of Schubert
polynomials in \S \ref{Cprelims} that $\chi(\partial_w(\CS_v(X_n)))$
is equal to $1$ when $w=v$ and equal to zero, otherwise.
\end{proof}

Following \cite[\S 5.1]{BKT1}, let
\[
\Gamma^{(n)}:= \Z[{}^nc_1, {}^nc_2,\ldots]
\]
be the ring of theta polynomials of level $n$. Notice that the
elements denoted by $\ti_r(x\,;y)$ in loc.\ cit.\
correspond to the generators ${}^nc_r$ here. According to
\cite[Thm.\ 2]{BKT1}, the single theta polynomials
${}^n\Ti_\la={}^n\Ti_\la(X)$ for all $n$-strict partitions $\la$ form
a $\Z$-basis of $\Gamma^{(n)}$. In the next result, the Weyl group
$W_n$ acts on the ring $\Gamma[X_n]$ in the usual way.

\begin{prop}
\label{invlem}
  The ring $\Gamma^{(n)}$ is equal to the subring $\Gamma[X_n]^{W_n}$ of
  $W_n$-invariants in $\Gamma[X_n]$.
\end{prop}
\begin{proof}
  We have $g\in \Gamma[X_n]^{W_n}$ if and only if $s_ig=g$ for all
  $i\in [0,n-1]$ if and only if $\partial_ig=0$ for
  $0\leq i \leq n-1$. Suppose that $f\in \Gamma[X_n]^{W_n}$ and employ
  Proposition \ref{basisprop} to write
\begin{equation}
\label{symmeq}
f(X_n) = \sum_{w\in W^{(n)}} a_w\CS_w(X_n).
\end{equation}
Applying the divided differences $\partial_i$ for $i\in [0,n-1]$ to
(\ref{symmeq}) and using (\ref{ddCeqinit}), we deduce that $a_w= 0$
for all $w\in W^{(n)}$ such that $\ell(ws_i)<\ell(w)$ for some
$i\in [0,n-1]$.  Therefore, $f$ is in the $\Z$-span of those
$\CS_w(X_n)$ for $w\in W^{(n)}$ with $\ell(ws_i)>\ell(w)$ for all
$i\in [0,n-1]$.  These are exactly the $n$-Grassmannian elements $w$
in $W_\infty$.  According to \cite[Prop.\ 6.2]{BKT1}, for any such
$w$, we have $\CS_w(X_n)={}^n\Ti_{\la(w)}(X)$ in $\Gamma[X_n]$.  It
follows that $f$ is a $\Z$-linear combination of theta polynomials of
level $n$, and hence that $f\in \Gamma^{(n)}$.  The converse is clear,
since $\partial_ih=0$ for all $i\in [0,n-1]$ and 
$h\in \Gamma^{(n)}$.
\end{proof}

\begin{example}
\label{Cex}
It follows from Proposition \ref{invlem} that
\[
\Gamma^{(n)}\cap \Z[X_n] = \Z[X_n]^{W_n} = \Z[e_1(X^2_n),\ldots,e_n(X^2_n)]
\]
where $X_n^2:=(x_1^2,\ldots,x_n^2)$. This can also be seen directly, using
the identities
\[
({}^nc_p)^2+2\sum_{i=1}^p(-1)^i\,({}^nc_{p+i})({}^nc_{p-i}) = e_p(X_n^2)
\]
for all $p\geq 0$ (compare with \cite[Eqn.\ (19)]{BKT1}).
\end{example}

Let $\IGam^{(n)}= \langle {}^nc_1, {}^nc_2,\ldots \rangle$ be the
ideal of $\Gamma[X_n]$ generated by the homogeneous elements in
$\Gamma^{(n)}$ of positive degree. For any parabolic subgroup $P$ of
$\Sp_{2n}$, let $\IGam^{(n)}_P$ be the corresponding ideal of
$\Gamma[X_n]^{W_P}$, and set $W_n^P:=W^P\cap W_n$. The following
result about the cohomology ring of $\Sp_{2n}/P$ is an immediate
consequence of Theorem \ref{Cthm}, Corollary \ref{GPcor} and the
discussion in \S \ref{dspC}.

\begin{cor}
\label{prescor}
There is a canonical ring isomorphism 
\[
\HH^*(\Sp_{2n}/B) \cong \Gamma[X_n]/\IGam^{(n)}
\]
which maps the cohomology class of the codimension $\ell(w)$ Schubert
variety $X_w$ to the class of the Schubert polynomial $\CS_w(X)$, for
any $w\in W_n$. Moreover, for any parabolic subgroup $P$ of
$\Sp_{2n}$, there is a canonical ring isomorphism
\[
\HH^*(\Sp_{2n}/P) \cong \Gamma[X_n]^{W_P}/\IGam^{(n)}_P
\]
which maps the cohomology class of the codimension $\ell(w)$ Schubert
variety $X_w$ to the class of the Schubert polynomial $\CS_w(X)$, for
any $w\in W^P_n$.
\end{cor}

\begin{example}
The version of Lemma \ref{Qprop} for single polynomials states that if
$\la$ is a strict partition of length $\ell$ and $p>\max(n,\la_1)$,
then $Q_{(p,\la)}\in \IGam^{(n)}$.  We can exhibit this containment more
explicitly as follows. For any integer $m>n$, we have
\[
c_m=\sum_{j=1}^{\infty}(-1)^{j-1}\,c_{m-j} {}^nc_j.
\]
This implies that for any integer vector $\al=(\al_1,\ldots,\al_\ell)$, 
the equality
\[
c_{(m,\al)}=\sum_{j=1}^\infty(-1)^{j-1}\,c_{(m-j,\al)} {}^nc_j
\]
holds, and therefore, by applying the Pfaffian operator $R^\infty$, that 
\begin{equation}
\label{Qpla}
Q_{(p,\la)}=\sum_{j=1}^\infty(-1)^{j-1}Q_{(p-j,\la)} {}^nc_j.
\end{equation}
It is important to notice that the terms $Q_{(p-j,\la)}$ in
(\ref{Qpla}) can be non-zero even when $p-j<0$. The `straightening law'
for such terms was found by Hoffman and Humphreys. For any integer $k$, let
$n(k):=\#\{i\ |\ \la_i>|k|\}$, and define the sets
\[
A_\la:=\{r\in [0,p-1]\ |\ r\neq \la_i \text{ for all $i\leq \ell$}\} \quad
\text{and} \quad 
B_\la:=\{\la_1,\ldots,\la_\ell\}.
\]
It follows from \cite[Thm.\ 9.2]{HH} that for any integer $k<p$, we have
\[
Q_{(k,\la)} = \begin{cases} (-1)^{n(k)} \,Q_{\la\cup k} & \text{if
    $k\in A_\la$}, \\ (-1)^{k+n(k)} \, 2\,Q_{\la\ssm |k|} & \text{if
    $|k|\in B_\la$}, \\ 0 & \text{otherwise},
\end{cases}
\]
where $\la\cup k$ and $\la\ssm |k|$ denote the partitions obtained by
adding (resp.\ removing) a part equal to $k$ (resp.\ $|k|$)
from $\la$.  Applying this in (\ref{Qpla}), we obtain
\begin{equation}
\label{Qneg}
Q_{(p,\la)}=\sum_{r\in A_\la}(-1)^{p-1-r+n(r)} \,Q_{\la\cup r} {}^nc_{p-r} 
+ 2\sum_{r\in B_\la} (-1)^{p-1+n(r)}\,Q_{\la\ssm r} {}^nc_{p+r}.
\end{equation}
The particular case of (\ref{Qneg}) when $(p,\la)=\delta_{n+1}$ reads
\[
Q_{\delta_{n+1}} = Q_{\delta_{n}}  {}^nc_{n+1} + 2
\sum_{r=1}^n(-1)^r\, Q_{\delta_n\ssm r} {}^nc_{n+1+r}.
\]
\end{example}

\subsection{The ring $\Gamma[X_n]$ as a $\Gamma^{(n)}$-module}
Set $e_p:=e_p(X_n)$ for each $p\in \Z$, and recall that
$e_\al:=\prod_ie_{\al_i}$ for any integer sequence $\al$.  Let $\cP_n$
denote the set of all strict partitions $\la$ with $\la_1\leq n$.

\begin{prop}
\label{freemod}
$\Gamma[X_n]$ is a free $\Gamma^{(n)}$-module of rank $2^nn!$ with basis
\[
\{e_\la(-X_n)x^\al\ |\ \la \in \cP_n, 
\ \ 0\leq \al_i \leq n-i, \ i\in [1,n]\}.
\]
\end{prop}
\begin{proof}
It is well known (see e.g \cite[(5.1${}'$)]{M2}) that $\Gamma[X_n]$ is
a free $\Gamma[e_1,\ldots,e_n]$-module with basis given by the
monomials $x^\al$ with $0\leq \al_i \leq n-i$ for $i\in [1,n]$. It
will therefore suffice to show that $\Gamma[e_1,\ldots,e_n]$ is a free
$\Gamma^{(n)}$-module with basis $e_\la(-X_n)$ for $\la\in
\cP_n$. Setting $y_j=0$ for $1\leq j\leq n$ in equation (\ref{tqeh})
gives
\[
E(X_n,t):=\sum_{p=0}^\infty e_p t^p=
\left(\sum_{p=0}^\infty {}^nc_p\,t^p\right)\left(\sum_{p=0}^\infty c_p(-t)^p\right).
\]
Using this and the relations (\ref{basicrels}), we obtain
\[
E(X_n,t)E(X_n,-t)=\left(\sum_{p=0}^\infty {}^nc_pt^r\right)
\left(\sum_{p=0}^\infty {}^nc_p(-t)^p\right)
\]
and therefore that
\[
e_p^2(-X_n) + 2\sum_{i=1}^p(-1)^ie_{p+i}(-X_n)e_{p-i}(-X_n) \in \Gamma^{(n)}
\]
for each $p \geq 1$.  It follows that the monomials $e_\la(-X_n)$ for
$\la\in\cP_n$ generate $\Gamma[e_1,\ldots,e_n]$ as a
$\Gamma^{(n)}$-module.  It remains to prove that these monomials
$e_\la(-X_n)$ are linearly independent over $\Gamma^{(n)}$.

We claim that the Schubert polynomials $\CS_w(X_n)$ for $w\in W_n$ are
linearly independent over $\Gamma^{(n)}$. Indeed, suppose that
\[
\sum_{w\in  W_n} f_w \CS_w(X_n) = 0 
\]
for some coefficients $f_w\in \Gamma^{(n)}$, and that $v\in W_n$ is an
element of maximal length such that $f_v\neq 0$. Then, by applying
(\ref{Cstar}), we have
\[
0 = \partial_v\left(\sum_{w\in  W_n} f_w \CS_w(X_n) \right) 
= f_v \partial_v(\CS_v(X_n)) = f_v,
\]
which is a contradiction, proving the claim. We have used here the
fact that the divided differences $\partial_i$ are
$\Gamma^{(n)}$-linear for each $i\in [0,n-1]$.

It follows that the Schur $Q$-polynomials $Q_\la=Q_\la(c)$ for $\la\in
\cP_n$ are linearly independent over $\Gamma^{(n)}$ (since these
are exactly the Schubert polynomials $\CS_w(X_n)$ which lie in $\Gamma$, 
with $w=w_\la\in W_n$). But the elements $\{Q_\la\}$ and $\{c_\la\}$ for 
$\la \in \cP_n$ are
related by an unitriangular change of basis matrix, and so are the 
elements $\{c_\la\}$ and $\{e_\la(-X_n)\}$. It follows that
the $Q_\la$ for $\la\in \cP_n$ generate
$\Gamma[e_1,\ldots,e_n]$ as a $\Gamma^{(n)}$-module, and hence that
the three aforementioned sets each form a basis. 
\end{proof}

Following \cite{PR}, for any partition $\la$, the $\wt{Q}$-polynomial 
is defined by
\begin{equation}
\label{Qtoe}
\wt{Q}_\la(X_n):=R^\infty\,e_\la(X_n).
\end{equation}

\begin{cor}
\label{basiscor}
The ring $\Gamma[X_n]$ is a free $\Gamma[X_n]^{S_n}$-module with basis
$\{\AS_\om(X)\}$ for $\om\in S_n$. The ring $\Gamma[X_n]^{S_n}$ is a
free $\Gamma^{(n)}$-module with basis $\{\wt{Q}_\la(-X_n)\}$ for
$\la\in\cP_n$.  The ring $\Gamma[X_n]$ is a free $\Gamma^{(n)}$-module
on the basis $\{\CS_w(X_n)\}$ of single type C Schubert polynomials
for $w\in W_n$, and is also free on the product basis
$\{\wt{Q}_\la(-X_n)\AS_\om(X)\}$ for $\la\in\cP_n$ and $\om\in S_n$.
\end{cor}
\begin{proof}
Since $\Gamma[X_n]^{S_n} = \Gamma[e_1,\ldots,e_n]$, the first
statement follows from Proposition \ref{freemod} and
\cite[(4.11)]{M2}. The assertions involving the polynomials
$\wt{Q}_\la(-X_n)$ are justified using Proposition \ref{freemod} and
equation (\ref{Qtoe}), and the fact that the Schubert polynomials
$\{\CS_w(X_n)\}$ for $w\in W_n$ form a basis is also clear.
\end{proof}

\subsection{A scalar product on $\Gamma[X_n]$}
Recall that $w_0=(\ov{1},\ldots,\ov{n})$ denotes the element of
longest length in $W_n$. If $f\in \Gamma[X_n]$, then
$\partial_i(\partial_{w_0}f)= 0$ for all $i$ with $0\leq i \leq n-1$.
Proposition \ref{invlem} implies that
$\partial_{w_0}(f)\in \Gamma^{(n)}$, for each $f\in \Gamma[X_n]$.

\begin{defn}
We define a scalar product $\langle\ ,\, \rangle$
on $\Gamma[X_n]$, with values in $\Gamma^{(n)}$,
by the rule 
\[
\langle f,g\rangle := \partial_{w_0}(fg), \ \ \ f,g\in \Gamma[X_n].
\]
\end{defn}

\begin{prop}
\label{2properties}
The scalar product $\langle\ ,\, \rangle: \Gamma[X_n]\times \Gamma[X_n]
\to \Gamma^{(n)}$ is $\Gamma^{(n)}$-linear. For any 
$f,g\in \Gamma[X_n]$ and $w\in W_n$, we have
\[
\langle\partial_wf,g \rangle  = \langle f,\partial_{w^{-1}}g\rangle.
\]
\end{prop}
\begin{proof}
  The scalar product is $\Gamma^{(n)}$-linear, since the same is true
  for the operator $\partial_{w_0}$. For the second statement, given
  $f,g\in \Gamma[X_n]$, it suffices to show that
  $\langle\partial_if,g \rangle = \langle f,\partial_ig\rangle$ for
  $0\leq i \leq n-1$. We have
\[
\langle\partial_if,g \rangle  = \partial_{w_0}((\partial_if)g) = 
\partial_{w_0s_i}\partial_i((\partial_if)g) = 
\partial_{w_0s_i}((\partial_if)(\partial_ig))
\]
because $s_i(\partial_if)=\partial_if$. The expression on the right is 
symmetric in $f$ and $g$, hence
\[
\langle\partial_if,g \rangle  = \langle\partial_ig,f \rangle =
\langle f,\partial_ig\rangle,
\]
as required.
\end{proof}

\begin{prop}
\label{prodprop}
Let $u,v\in W_n$ be such that $\ell(u)+\ell(v)=n^2$. Then we have
\[
\langle\CS_u(X_n),\CS_v(X_n)\rangle = 
\begin{cases}
1 & \text{if $v=w_0u$}, \\
0 & \text{otherwise}.
\end{cases}
\]
\end{prop}
\begin{proof}
Using (\ref{Cstar}) and Proposition \ref{2properties}, we obtain
\[
\langle\CS_u(X_n),\CS_v(X_n)\rangle = \langle \partial_{u^{-1}w_0}
\CS_{w_0}(X_n), \CS_v(X_n)
\rangle = \langle \CS_{w_0}(X_n), \partial_{w_0u}\CS_v(X_n) \rangle.
\]
Also $\ell(w_0u)= \ell(w_0)-\ell(u)= \ell(v)$, and we deduce that
\[
\partial_{w_0u}\CS_v(X_n)  = 
\begin{cases}
1 & \text{if $v = w_0u$}, \\ 
0 & \text{otherwise}.
\end{cases}
\]
Since $\langle \CS_{w_0}(X_n), 1 \rangle = \partial_{w_0}(\CS_{w_0}(X_n)) = 1$,
the result follows.
\end{proof}

Although the elements of the $\Gamma^{(n)}$-basis
$\{\wt{Q}_\la(-X_n)\AS_\om(X)\}$ of $\Gamma[X_n]$ do not represent the
Schubert classes on the symplectic flag manifold, this product basis
is convenient for computational purposes. Indeed, following Lascoux
and Pragacz \cite{LP1} (in the finite case), one can identify
the dual $\Gamma^{(n)}$-basis of $\Gamma[X_n]$ relative to the scalar
product $\langle\ ,\, \rangle$, by working as shown below.

Let $\om_0=(n,n-1,\ldots,1)$ denote the permutation of  
longest length in $S_n$, and define $v_0:=w_0\om_0=\om_0w_0$.
We have
\begin{equation}
\label{wvcomm}
\partial_{w_0} = \partial_{v_0}\partial_{\om_0} = \partial_{\om_0}
\partial_{v_0}.
\end{equation}
We define a $\Gamma[X_n]^{S_n}$-valued scalar product $(\ ,\, )$
on $\Gamma[X_n]$ by the rule 
\[
(f,g) := \partial_{\om_0}(fg), \ \ \ f,g\in \Gamma[X_n].
\]
According to \cite[(5.12)]{M2}, the Schubert polynomials $\AS_u(X)$
for $u\in S_n$ satisfy the orthogonality relation
\[
\left(\,\AS_u(X), \,\om_0\AS_{u'\om_0}(-X)\,\right) = \delta_{u,u'}
\]
for any $u,u'\in S_n$.

Furthermore, define a $\Gamma^{(n)}$-valued scalar product $\{\ ,\, \}$
on $\Gamma[X_n]^{S_n}$ by the rule 
\[
\{f,g\} := \partial_{v_0}(fg), \ \ \ f,g\in \Gamma[X_n]^{S_n}.
\]
According to \cite[Thm.\ 5.23]{PR}, for any two partitions $\la,\mu\in
\cP_n$, we have
\[
\left\{\wt{Q}_\la(-X_n),\wt{Q}_{\delta_n\ssm \mu}(-X_n)\right\} = \delta_{\la,\mu},
\]
where $\delta_n\ssm \mu$ is the strict partition whose parts complement
the parts of $\mu$ in the set $\{n,n-1,\ldots,1\}$, and 
$\delta_{\la,\mu}$ denotes the Kronecker delta.

Observe that $(\ ,\, )$ is $\Gamma[X_n]^{S_n}$-linear and 
$\{\ ,\, \}$ is $\Gamma^{(n)}$-linear. Then (\ref{wvcomm}) gives 
\[
\langle f,g \rangle = \{(f,g)\}, \ \ \ \text{for any $f,g\in\Gamma[X_n]$},
\]
and moreover the orthogonality relation
\[
\left\langle\wt{Q}_\la(-X_n)\AS_u(X), \wt{Q}_{\delta_n\ssm \mu}(-X_n)
(\om_0\AS_{u'\om_0}(-X)) \right\rangle =\delta_{u,u'}\delta_{\la,\mu}
\]
holds, for any $u,u'\in S_n$ and $\la,\mu\in \cP_n$.  The reader should
compare this to the discussion in \cite[\S 1]{LP1}.

\section{Double Schubert polynomials of types B and D}
\label{dspBD}

\subsection{Preliminaries}
\label{Dprelims}
Let $b:=(b_1,b_2,\ldots)$ be a sequence of commuting variables, and
set $b_0:=1$ and $b_p=0$ for $p<0$. Consider the graded ring $\Gamma'$
which is the quotient of the polynomial ring $\Z[b]$ modulo the ideal
generated by the relations
\[
b_p^2+2\sum_{i=1}^{p-1}(-1)^i b_{p+i}b_{p-i}+(-1)^p b_{2p}=0, \ \ \ 
\text{for all $p\geq 1$}.
\]
The ring $\Gamma'$ is isomorphic to the ring of Schur $P$-functions.
Following \cite{P}, the $P$-functions map naturally to the Schubert
classes on maximal (odd or even) orthogonal Grassmannians. We regard
$\Gamma$ as a subring of $\Gamma'$ via the injective ring homomorphism
which sends $c_p$ to $2b_p$ for every $p\geq 1$.

The Weyl group for the root system of type $\text{B}_n$ is the same
group $W_n$ as the one for type $\text{C}_n$. The
Ikeda-Mihalcea-Naruse type B double Schubert polynomials $\BS_w(X,Y)$
for $w\in W_\infty$ form a natural $\Z[Y]$-basis of
$\Gamma'[X,Y]$. For any Weyl group element $w$, the polynomial
$\BS_w(X,Y)$ satisfies
\[
\BS_w(X,Y) = 2^{-s(w)} \CS_w(X,Y), 
\]
where $s(w)$ denotes the number of indices $i$ such that $w_i<0$.  The
algebraic theory of these polynomials is thus nearly identical to that
in type C, provided that one uses coefficients in the ring $\Z[\frac{1}{2}]$. 

If $\BS'_w =\BS'_w(X,Y)$ is the polynomial obtained from $\BS_w(X,Y)$
by setting $x_j=y_j=0$ for all $j>n$, then the $\BS'_w$ for $w\in
W^{(n)}$ form a $\Z[Y_n]$-basis of $\Gamma'[X_n,Y_n]$.  The
polynomials $\BS'_w$ for $w\in W_n$ represent the equivariant Schubert
classes on the odd orthogonal flag manifold $\SO_{2n+1}/B$, whose
equivariant cohomology ring (with $\Z[\frac{1}{2}]$-coefficients) is
isomorphic to that of $\Sp_{2n}/B$. For further details, the reader
may consult the references \cite{IMN1} and \cite[\S 6.3.1]{T5}.

In the rest of this paper we discuss the corresponding theory for the
even orthogonal group, that is, in Lie type D, and assume that $n\geq 2$.
The Weyl group $\wt{W}_n$ for the root system $\text{D}_n$ is the
subgroup of $W_n$ consisting of all signed permutations with an even
number of sign changes.  The group $\wt{W}_n$ is an extension of $S_n$
by the element $s_\Box=s_0s_1s_0$, which acts on the right by
\[
(w_1,w_2,\ldots,w_n)s_\Box=(\ov{w}_2,\ov{w}_1,w_3,\ldots,w_n).
\]
There is a natural embedding $\wt{W}_k\hookrightarrow \wt{W}_{k+1}$ of
Weyl groups defined by adjoining the fixed point $k+1$, and we let
$\wt{W}_\infty := \cup_k \wt{W}_k$. The elements of the set $\N_\Box
:=\{\Box,1,\ldots\}$ index the simple reflections in $\wt{W}_\infty$.
The {\em length} $\ell(w)$ of an element $w\in \wt{W}_\infty$ is defined 
as in type C. The element of longest length $\wt{w}_0=\wt{w}_0^{(n)}$ 
in $\wt{W}_n$ satisfies
\[
\wt{w}_0=\begin{cases}
         (\ov{1},\ldots,\ov{n}) & \text{if $n$ is even}, \\
         (1,\ov{2},\ldots,\ov{n}) & \text{if $n$ is odd}. 
         \end{cases} 
\]

We define an action of $\wt{W}_\infty$ on $\Gamma'[X,Y]$ by ring
automorphisms as follows. The simple reflections $s_i$ for $i>0$ act
by interchanging $x_i$ and $x_{i+1}$ and leaving all the remaining
variables fixed. The reflection $s_\Box$ maps 
$(x_1,x_2)$ to $(-x_2,-x_1)$, fixes the $x_j$ for $j\geq 3$ and all 
the $y_j$, and satisfies, for any $p\geq 1$,
\begin{align*}
\label{s_Box}
s_\Box(b_p) :=
b_p+(x_1+x_2)\sum_{j=0}^{p-1}\left(\sum_{a+b=j}x_1^ax_2^b\right)
c_{p-1-j}.
\end{align*}
For each $i\in \N_\Box$, define the divided difference operator
$\partial_i^x$ on $\Gamma'[X,Y]$ by
\[
\partial_\Box^xf := \frac{f-s_\Box f}{-x_1-x_2}, \qquad
\partial_i^xf := \frac{f-s_if}{x_i-x_{i+1}} \ \ \ \text{for $i\geq 1$}.
\]
Consider the ring involution $\omega:\Gamma'[X,Y]\to\Gamma'[X,Y]$
determined by
\[
\omega(x_j) = -y_j, \qquad
\omega(y_j) = -x_j, \qquad
\omega(b_p)=b_p
\]
and set $\partial_i^y:=\omega\partial_i^x\omega$ for each $i\in \N_\Box$.

The Ikeda-Mihalcea-Naruse double Schubert polynomials
$\DS_w=\DS_w(X,Y)$ for $w\in \wt{W}_{\infty}$ are the unique family of
elements of $\Gamma'[X,Y]$ satisfying the equations
\begin{equation}
\label{Dddeq}
\partial_i^x\DS_w = \begin{cases}
\DS_{ws_i} & \text{if $\ell(ws_i)<\ell(w)$}, \\ 
0 & \text{otherwise},
\end{cases}
\quad
\partial_i^y\DS_w = \begin{cases}
\DS_{s_iw} & \text{if $\ell(s_iw)<\ell(w)$}, \\ 
0 & \text{otherwise},
\end{cases}
\end{equation}
for all $i\in \N_\Box$, together with the condition that the constant
term of $\DS_w$ is $1$ if $w=1$, and $0$ otherwise.

The operators $\partial_i:=\partial_i^x$ for $i\in \N_\Box$ satisfy the
same Leibnitz rule (\ref{LeibR}) as in the type C case, and for any
$w\in\wt{W}_\infty$, the divided difference operator $\partial_w$ is
defined as before. For any $u,w\in \wt{W}_\infty$, we have
\[
\partial_u\DS_w(X,Y) = \begin{cases}
\DS_{wu^{-1}}(X,Y) & \text{if $\ell(wu^{-1}) = \ell(w) - \ell(u)$}, \\
0 & \text{otherwise}.
\end{cases}
\]

\subsection{The set $\wt{W}^{(n)}$ and the ring $\Gamma'[X_n,Y_n]$}
It is known that the $\DS_w$ for $w\in \wt{W}_\infty$ form a
$\Z[Y]$-basis of $\Gamma'[X,Y]$. Let $\DS^{(n)}_w =\DS^{(n)}_w(X_n,Y_n)$ be the
polynomial obtained from $\DS_w(X,Y)$ by setting $x_j=y_j=0$ for all
$j>n$.  For every $n\geq 1$, let
$$\wt{W}^{(n)}:=\{w\in \wt{W}_\infty\ |\ w_{n+1}< w_{n+2}<\cdots\}.$$
Let $\DS_w(X):=\DS_w(X,0)$ denote the single Schubert polynomial.

\begin{prop}
\label{basispropD}
The $\DS_w(X)$ for $w\in \wt{W}^{(n)}$ form a $\Z$-basis of
$\Gamma'[X_n]$, and a $\Z[Y]$-basis of $\Gamma'[X_n,Y]$. The $\DS^{(n)}_w$
for $w\in \wt{W}^{(n)}$ form a $\Z[Y_n]$-basis of $\Gamma'[X_n,Y_n]$.
\end{prop}
\begin{proof}
The argument is the same as for the proofs of Propositions 
\ref{basisprop}, \ref{basisprop2}, and Corollary \ref{basiscor2}
in \S \ref{dspC}.
\end{proof}

\subsection{The geometrization map $\pi'_n$}
\label{Dgeom}
The double Schubert polynomials $\DS^{(n)}_w(X,Y)$ for $w\in \wt{W}_n$
represent the equivariant Schubert classes on the even orthogonal flag
manifold.  Let $\{e_1,\ldots,e_{2n}\}$ denote the standard orthogonal
basis of $E:=\C^{2n}$ and let $F_i$ be the subspace spanned by the
first $i$ vectors of this basis, so that $F_{n-i}^\perp = F_{n+i}$ for
$0\leq i \leq n$. We say that two maximal isotropic subspaces $L$ and
$L'$ of $E$ are {\em in the same family} if $\dim(L\cap L')\equiv n$
(mod 2).  The orthogonal flag manifold $\M_n'$ parametrizes complete
flags $E_\bull$ in $E$ with $E_{n-i}^\perp = E_{n+i}$ for $0\leq i
\leq n$, and $E_n$ in the same family as $\langle
e_{n+1},\ldots,e_{2n}\rangle$.  Equivalently, $E_n$ is in the same
family as $F_n$, if $n$ is even, and in the opposite family, if $n$ is
odd. We have that $\M_n'=\SO_{2n}/B$ for a Borel subgroup $B$ of the
orthogonal group $\SO_{2n}=\SO_{2n}(\C)$. If $T$ denotes the
associated maximal torus in $B$, then the $T$-equivariant cohomology
ring $\HH^*_{T}(\M'_n)$ is a $\Z[Y_n]$-algebra, where $y_i$ is
identified with the equivariant Chern class
$-c_1^T(F_{n+1-i}/F_{n-i})$, for $1\leq i \leq n$.

The Schubert varieties in $\M_n'$ are the closures of the
$B$-orbits, and are indexed by the elements of
$\wt{W}_n$. Concretely, any $w\in \wt{W}_n$ corresponds to a Schubert
variety $X_w=X_w(F_\bull)$ of codimension $\ell(w)$, which is the 
closure of the $B$-orbit
\[
   X^\circ_w := \{ E_\bull \in \M'_n \mid \dim(E_r \cap
   F_s) = d'_w(r,s) \ \, \mathrm{for} \ 1\leq r \leq n-1, \, 1\leq s \leq 2n \},
\]
where $d'_w(r,s)$ denotes the rank function defined as follows. There
is a group monomorphism $\zeta:\wt{W}_n\hra S_{2n}$, defined by
restricting the map $\zeta$ of Section \ref{geommap} to $\wt{W}_n$.
Then $d'_w(r,s)$ equals the number of $i\leq r$ such that
$\zeta(\wt{w}_0w\wt{w}_0)_i > 2n-s$. Since $X_w$ is stable under the
action of $T$, we obtain an {\em equivariant Schubert class}
$[X_w]^T:=[ET\times^{T}X_w]$ in $\HH^*_{T}(\M'_n)$.

Following \cite{IMN1}, there is a surjective homomorphism of graded 
$\Z[Y_n]$-algebras
\[
\pi'_n:\Gamma'[X_n,Y_n] \to \HH^*_{T}(\M'_n)
\]
such that 
\begin{equation}
\label{Dweq}
\pi'_n(\DS^{(n)}_w)= \begin{cases}
[X_w]^T & \text{if $w\in \wt{W}_n$}, \\
0 & \text{if $w\in \wt{W}^{(n)}\smallsetminus \wt{W}_n$}.
\end{cases}
\end{equation}
We let $E_i$ denote the $i$-th tautological vector vector bundle over
$\M'_n$, for $0\leq i \leq 2n$.  The map $\pi'_n$ is defined by the equations
\begin{equation}
\label{Dweq2}
\pi'_n(x_i) = c_1^T(E_{n+1-i}/E_{n-i}) \ \ \text{and} \ \ 
\pi'_n(b_p) = \frac{1}{2}\,c_p^T(E-E_n-F_n)
\end{equation}
for $1\leq i \leq n$ and $p\geq 1$.

\begin{remark}
The above convention on the family of $E_n$ in the definition of
$\M_n'$ differs from that stated in \cite[\S 6.3.2]{T5} and \cite[\S
4.1]{T6}, and corrects these latter two references. This is
necessary in order for the formulas (\ref{Dweq}) and (\ref{Dweq2}) to
hold, which are directly analogous to the ones for the Lie types B and
C.
\end{remark}

\subsection{The kernel of the map $\pi'_n$}
In the following discussion, it suffices to work with coefficients in
the ring $\Z[\frac{1}{2}]$, but for ease of notation we will employ
the rational numbers $\Q$ instead. For any abelian group $A$, let
$A_\Q:=A\otimes_\Z\Q$, and use the tensor product to extend $\pi'_n$
to a homomorphism of $\Q[Y_n]$-algebras
\[
\pi'_n:\Gamma'[X_n,Y_n]_\Q \to \HH^*_{T}(\M'_n)_\Q.
\]

\begin{defn}
Define
\[
\wt{b}_n:=\sum_{j=0}^{n-1}b_{n-j}e_j(-Y_n),
\]
let 
\[
\wh{B}^{(n)}:= \Z[\wt{b}_n, {}^nc^n_1, {}^nc^n_2,\ldots],
\] 
and let
$\wh{\IB}^{(n)}$ be the ideal of $\Gamma'[X_n,Y_n]_\Q$ generated by
the homogeneous elements in $\wh{B}^{(n)}$ of positive degree.
\end{defn}

\begin{lemma}
\label{OGevenlem}
We have $\wh{\IB}^{(n)}\subset \Ker\pi'_n$.
\end{lemma}
\begin{proof}
Let $A_n:=(a_1,\ldots,a_n)$ and $\HH_n:=\Q[A_n,Y_n]/L_n$, where $L_n$
is the ideal of $\Q[A_n,Y_n]$ generated by the differences
$e_i(A_n^2)-e_i(Y_n^2)$ for $1 \leq i \leq n-1$ and the difference
$e_n(A_n)-e_n(-Y_n)$.  It is known that the equivariant cohomology
ring $\HH_{T}^*(\M'_n)_\Q$ is canonically isomorphic to $\HH_n$ as a
$\Q[Y_n]$-algebra (compare with \cite[\S 3]{F2}). The
geometrization map $\pi'_n:\Gamma'[X_n,Y_n]_\Q\to \HH_n$ satisfies
$\pi'_n(x_j)=-a_j$ for $1\leq j \leq n$, while
\[
\pi'_n(b_p):=\frac{1}{2}\sum_{i+j=p}e_i(A_n)h_j(Y_n), \ \ \ p\geq 1.
\]
The element $e_n(A_n)-e_n(-Y_n)$ is thus identified with the
difference $(-1)^nc^T_n(E_n) - c^T_n(F_n)$. Our conventions on the
families of $E_n$ and $F_n$ imply that the latter class vanishes in
$\HH_{T}^*(\M'_n)_\Q$, by a result of Edidin and Graham
\cite[Thm. 1]{EG}.

We deduce that ${}^nc^n_p \in \Ker\pi'_n$ for each
$p\geq 1$ as in the proof of Lemma \ref{IGlem}, 
so it suffices to check that $\wt{b}_n\in \Ker\pi'_n$. Indeed, we have
\begin{gather*}
\pi_n'(2\wt{b}_n) = \sum_{j=0}^{n-1}e_j(-Y_n)\sum_{\alpha+\beta=n-j}
e_\alpha(A_n)h_\beta(Y_n) \\
= \sum_{\alpha=1}^ne_\alpha(A_n)\sum_{j=0}^{n-\alpha}e_j(-Y_n)h_{n-\alpha-j}(Y_n)
+\sum_{j=0}^{n-1}e_j(-Y_n)h_{n-j}(Y_n) \\
= e_n(A_n) - e_n(-Y_n).
\end{gather*}
\end{proof}

If $\la=(\la_1>\la_2>\cdots > \la_\ell)$ is a strict partition, let
$\wt{w}_\la$ be the corresponding increasing element of
$\wt{W}_\infty$, so that the negative components of $\wt{w}_\la$ are
exactly $-\la_1-1,\ldots,-\la_\ell-1$ and possibly also $-1$,
depending on the parity of $\ell=\ell(\la)$.

\begin{lemma}
\label{Pprop}
If $\la$ is a strict partition with $\la_1\geq n$, then 
$\DS^{(n)}_{\wt{w}_\la}(X_n,Y_n)\in \wh{\IB}^{(n)}$.
\end{lemma}
\begin{proof}
  For each strict partition $\mu$ of length $\ell$, let
  $P_\mu:= \DS^{(n)}_{\wt{w}_\mu}(X_n,Y_n)$.  According to \cite[Thm.\
  6.6]{IMN1} and \cite[(2.11)]{IMN2}, we have the Pfaffian recursion
\begin{equation}
\label{PPfeq}
P_\mu = \sum_{j=2}^\ell(-1)^j P_{\mu_1,\mu_j}
P_{\mu_2,\ldots,\wh{\mu}_j,\ldots, \mu_\ell}.
\end{equation}
Moreover, it follows from \cite[Prop.\ 2.1]{IMN2} that every factor
$P_{\mu_1,\mu_j}$ in (\ref{PPfeq}) is a $\Z[Y_n]$-linear combination
of products $P_rP_s$ with $r\geq \mu_1$.

It is easy to show that for any integer $r\geq 1$,
\[
P_r = \begin{cases}
  \sum_{j=0}^{r-1} b_{r-j} e_j(-Y_r) & \text{if $r\leq n$}, \\
   \sum_{j=0}^n b_{r-j} e_j(-Y_n) & \text{if $r>n$}.
\end{cases}
\]
It follows that $P_r = \frac{1}{2} \ov{Q}^{1-r}_r$ for all $r>n$,
while $P_n=\wt{b}_n$. We now deduce from equation (\ref{Qleq}) and
Lemma \ref{Qprop} that $P_r\in \wh{\IB}^{(n)}$ for every $r\geq
n$. The proof is finished by combining this fact with (\ref{PPfeq}).
\end{proof}

\begin{lemma}
\label{DSlem}
For any $w\in \wt{W}_{\infty}\ssm \wt{W}_n$, we have $\DS^{(n)}_w\in \wh{\IB}^{(n)}$.
\end{lemma}
\begin{proof}
Let $w$ be an element of $\wt{W}_\infty$ and $i<j$. Following
\cite[Lemma 2]{B}, we have $\ell(w\ov{t}_{ij})= \ell(w)+1$ if and only
if (i) $-w_i<w_j$, and (ii) there is no $p<i$ such that
$-w_j<w_p<w_i$, and no $p<j$ such that $-w_i < w_p < w_j$.

The group $\wt{W}_\infty$ acts on the polynomial ring
$\Z[y_1,y_2,\ldots]$, with $s_i$ for $i \geq 1$ interchanging $y_i$
and $y_{i+1}$ and leaving all the remaining variables fixed, and
$s_\Box$ mapping $(y_1,y_2)$ to $(-y_2,-y_1)$ and fixing the $y_j$
with $j\geq 3$.  Let $w\in \wt{W}_\infty$ be non-increasing, let $r$
be the last positive descent of $w$, let $s:=\max(i>r\ |\ w_i<w_r)$,
and let $v:= wt_{rs}$.  According to \cite[Prop.\ 6.12]{IMN1}, the
double Schubert polynomials $\DS_u=\DS_u(X,Y)$ satisfy the {\em
  transition equations}
\begin{equation}
\label{Dtrans}
\DS_w = (x_r-v(y_r))\DS_v + \sum_{{1\leq i < r} \atop
{\ell(vt_{ir}) = \ell(w)}} \DS_{vt_{ir}} + 
\sum_{{i\geq 1, i\neq r} \atop {\ell(v\ov{t}_{ir}) = 
\ell(w)}} \DS_{v\ov{t}_{ir}}
\end{equation}
in $\Gamma'[X,Y]$.  The recursion (\ref{Dtrans}) terminates in a
$\Z[X,Y]$-linear combination of elements $\DS_{\wt{w}_\nu}(X,Y)$ for
strict partitions $\nu$.

For any $w\in \wt{W}_\infty$, let $\mu(w)$ denote the strict partition
whose parts are the elements of the set
$\{|w_i|-1 \ :\ w_i<0\}$.  Clearly we have $\mu(w)=\mu(wu)$ for any $u\in
S_\infty$. In equation (\ref{Dtrans}), we therefore have
$\mu(v)=\mu(vt_{ir}) = \mu(w)$. Moreover, condition (i) above 
shows that the parts of $\mu(v\ov{t}_{ir})$ are greater
than or equal to the parts of $\mu(w)$. In particular, if
$\mu(w)_1\geq n$, then $\mu(v\ov{t}_{ir})_1 \geq n$.

Assume first that $w\in \wt{W}_{n+1}\ssm \wt{W}_n$.
If $w_i=-n-1$ for some $i\leq n+1$,
we use the transition recursion (\ref{Dtrans}) to write $\DS^{(n)}_w$ as a 
$\Z[X_n,Y_n]$-linear combination of elements $\DS^{(n)}_{\wt{w}_\nu}$ for strict 
partitions $\nu$ with $\nu_1\geq n$. Lemma \ref{Pprop}
now implies that $\DS^{(n)}_w\in \wh{\IB}^{(n)}$.

We next suppose that $w_i=n+1$ for some $i\leq n$. Let 
$$\{v_2,\ldots, v_n\} := \{w_1,\ldots,\wh{w_i},\ldots, w_n\}$$ with 
$v_2>\cdots > v_n$, and define
\[
u:=(\ov{v_2},n+1,v_3,\ldots,v_n,w_{n+1})\in \wt{W}_{n+1}
\]
and 
\[
\ov{u}:=us_\Box=(\ov{n+1},v_2,v_3,\ldots,v_n,w_{n+1}).
\]
Then we have $\DS^{(n)}_{\ov{u}}\in \wh{\IB}^{(n)}$ from the previous
case, and $\partial_\Box(\DS^{(n)}_{\ov{u}})= \DS^{(n)}_u$.

For any $i$ such that $\Box\leq i \leq n-1$, it is easy to verify that
$s_i({}^nc_p^n)={}^nc_p^n$ and $s_i(\wt{b}_n)=\wt{b}_n$, and hence
that $\partial_i({}^nc_p^n)=\partial_i(\wt{b}_n)=0$.  We therefore
obtain that $\partial_i(\wh{\IB}^{(n)})\subset \wh{\IB}^{(n)}$ for all
$i\in [\Box,n-1]$. It follows that $\DS^{(n)}_u\in \wh{\IB}^{(n)}$, and the
proof is now concluded in the same way as in Lemma \ref{CSlem}.
\end{proof}

\begin{thm}
\label{Dthm}
Let $J'_n:=\bigoplus_{w\in \wt{W}^{(n)}\ssm \wt{W}_n}\Q[Y_n]\DS^{(n)}_w$. Then 
we have 
\[
\wh{\IB}^{(n)} = J'_n=\sum_{w\in \wt{W}_{\infty}\ssm \wt{W}_n}
\Q[Y_n]\DS^{(n)}_w= \Ker \pi'_n.
\]
We have a canonical isomorphism of $\Q[Y_n]$-algebras
\[
\HH_{T}^*(\SO_{2n}/B,\Q) \cong \Gamma'[X_n,Y_n]_\Q/\wh{\IB}^{(n)}.
\]
\end{thm}
\begin{proof}
The argument is the same as the proof of Theorem \ref{Cthm}, this 
time using Lemma \ref{OGevenlem}, Lemma \ref{DSlem}, and
Proposition \ref{basispropD}.
\end{proof}

\subsection{Partial even orthogonal flag manifolds}

We can generalize the presentation in Theorem \ref{Dthm} to the
partial flag manifolds $\SO_{2n}/P$, where $P$ is a parabolic subgroup
of $\SO_{2n}$. The parabolic subgroups $P$ containing $B$ correspond
to sequences $\fraka \ :\ a_1<\cdots < a_p$ of elements of $\N_{\Box}$
with $a_p<n$. The manifold $\SO_{2n}/P$ parametrizes partial flags of
subspaces
\[
0 \subset E_1 \subset \cdots \subset E_p \subset E=\C^{2n}
\]
with $\dim(E_j) = n-a_{p+1-j}$ for each $j\in [1,p]$ and $E_p$ isotropic. 
If $a_1=\Box$, so that $\dim(E_p)=n$, then we insist that the family of 
$E_p$ obeys the same convention as in Section \ref{Dgeom}.

A sequence $\fraka$ as above parametrizes the parabolic
subgroup $\wt{W}_P$ of $\wt{W}_n$, which is generated by the simple reflections
$s_i$ for $i\notin\{a_1,\ldots, a_p\}$. Let $\Gamma'[X_n,Y_n]_\Q^{\wt{W}_P}$ be
the subring of elements in $\Gamma'[X_n,Y_n]_\Q$ which are fixed by 
the action of $\wt{W}_P$, i.e.,
\[
\Gamma'[X_n,Y_n]_\Q^{\wt{W}_P} = \{ f\in \Gamma'[X_n,Y_n]_\Q\ |\ 
s_i(f)=f, \ \forall\, i \notin\{a_1,\ldots, a_p\}, \ i<n\}.
\]
Then $\Gamma'[X_n,Y_n]_\Q^{\wt{W}_P}$ is a $\Q[Y_n]$-subalgebra of
$\Gamma'[X_n,Y_n]_\Q$. Let $\wt{W}^P\subset \wt{W}^{(n)}$ denote the set
\[
\wt{W}^P := \{w\in \wt{W}^{(n)}\ |\ \ell(ws_i) = \ell(w)+1,\  \forall\, i \notin
\{a_1,\ldots, a_p\}, \ i<n\}.
\]
Then by arguing as in Section \ref{psfms}, we obtain the following two results.

\begin{prop}
\label{wppropD}
We have 
\begin{equation}
\label{topreqD}
\Gamma[X_n,Y_n]_\Q^{\wt{W}_P} = \bigoplus_{w\in \wt{W}^P} \Q[Y_n]\DS^{(n)}_w.
\end{equation}
\end{prop}

\begin{cor}
\label{GPcorD}
There is a canonical isomorphism of $\Q[Y_n]$-algebras
\[
\HH_{T}^*(\SO_{2n}/P,\Q) \cong \Gamma'[X_n,Y_n]_\Q^{\wt{W}_P}/\wh{\IB}^{(n)}_P
\]
where $\wh{\IB}^{(n)}_P$ is the ideal of $\Gamma'[X_n,Y_n]_\Q^{W_P}$
generated by the homogeneous elements in $\wh{B}^{(n)}$ of
positive degree.
\end{cor}

\section{Divided differences and double eta polynomials}
\label{ddsdbleta}

\subsection{Preliminaries}
Fix $k\geq 0$, and set ${}^kc_p={}^kc_p(X):= \sum_{i=0}^p
c_{p-i}h^{-k}_i(X)$. Define ${}^kb_p := {}^kc_p$ for $p<k$, ${}^kb_p
:= \frac{1}{2}{}^kc_p$ for $p>k$, and set
\[
{}^kb_k := \frac{1}{2}{}^kc_k + \frac{1}{2}e^k_k(X) \quad \text{and} \quad
{}^k\wt{b}_k := \frac{1}{2}{}^kc_k - \frac{1}{2}e^k_k(X).
\]
Let $f_k$ be an indeterminate of degree $k$, which will equal
${}^kb_k$, ${}^k\wt{b}_k$, or $\frac{1}{2}\,{}^kc_k$ in the sequel.
We also let $f_0 \in\{0,1\}$. For any $p,r\in \Z$, define
${}^k\wh{c}_p^r$ by
\[
{}^k\wh{c}_p^r:= {}^kc_p^r + 
\begin{cases}
(2f_k-{}^kc_k)e^{p-k}_{p-k}(-Y) & \text{if $r = k - p < 0$}, \\
0 & \text{otherwise}.
\end{cases}
\]

It is easy to see that $\omega({}^kc^r_p)= {}^{-r}c_p^{-k}$ for any $k,r\in \Z$, 
and if $r\leq 0 \leq k$, then 
\[
\omega({}^k\wh{c}^r_p)= {}^{-r}\wh{c}_p^{-k}.
\]
We now have the following even orthogonal analogues of Lemmas
\ref{ddylem} and \ref{imlem1dual}, which are dual versions of results
from \cite[Prop.\ 1 and Prop.\ 2]{T6} and \cite{T7}.

\begin{lemma}
\label{ddylemD}
Suppose that $k,p,r\in \Z$ and $i\geq 1$.

\medskip
\noin
{\em (a)} We have 
\[
\partial_i ({}^kc_p^r)= 
\begin{cases}
{}^{k-1}c_{p-1}^r & \text{if $k=\pm i$}, \\
0 & \text{otherwise}.
\end{cases}
\]
\medskip
\noin
{\em (b)} If $p>k\geq 0$, we have 
\[
\partial_i({}^k\wh{c}_p^{k-p}) =
\begin{cases}
  {}^{k-1}\wh{c}_{p-1}^{k-p} & \text{if $i=p-k\geq 2$}, \\
  2\omega(f_k) & \text{if $i=p-k=1$}, \\
0 & \text{otherwise}.
\end{cases}
\]
\end{lemma}

\begin{lemma}
\label{Dddlem2dual}
Suppose that $k,p,r\in \Z$ and $r\leq 0$. We have 
\[
\partial_\Box \left({}^kc_p^r\right)= 
\begin{cases}
{}^{-2}c_{p-1}^r & \text{if $k=-1$}, \\
2\left({}^{-2}c_{p-1}^r\right) & \text{if $k=0$}, \\
2\left({}^{-1}c_{p-1}^r\right) -{}^0c^r_{p-1} & \text{if $k=1$}, \\
0 & \text{if $|k|\geq 2$}.
\end{cases}
\]
\end{lemma}

 For $s\in \{0,1\}$, define
\[
f_k^s:= f_k+\sum_{j=1}^k {}^kc_{k-j}h_j^s(-Y), 
\]
set $\wt{f}_k:={}^kc_k-f_k$ and $\wt{f}_k^s:={}^kc_k-2f_k+f_k^s$.

\begin{lemma}
\label{ddylemD2}
Suppose that $k,p\in \Z$ with $p>k$. Then we have
\[
\partial_\Box \left({}^k\wh{c}_p^{k-p}\right)= 
\begin{cases}
2\omega(\wt{f}^1_k) & \text{if $k-p=-1$}, \\
0 & \text{if $k-p < -1$}.
\end{cases}
\]
\end{lemma}

\begin{lemma}
\label{Dlem1dual}
Suppose that $k\geq 0$ and $r\geq 1$. Then we have
\[
{}^k\wh{c}_p^{-r} = {}^{k+1}\wh{c}_p^{-r+1}-(x_{k+1}+y_r)\, {}^k\wh{c}_{p-1}^{-r+1}.
\]
\end{lemma}

Let $\rho$ be a composition and let $\al=(\al_1,\ldots,\al_\ell)$ and
$\be=(\be_1,\ldots,\be_\ell)$ be two integer vectors. Define
\[
{}^\rho\wh{c}_\al^\be:= {}^{\rho_1}\wh{c}_{\al_1}^{\be_1}
{}^{\rho_2}\wh{c}_{\al_2}^{\be_2}\cdots
\]
where, for each $i\geq 1$, 
\[
{}^{\rho_i}\wh{c}_{\al_i}^{\be_i}:= {}^{\rho_i}c_{\al_i}^{\be_i} + 
\begin{cases}
(-1)^ie^{\rho_i}_{\rho_i}(X)e^{\al_i-\rho_i}_{\al_i-\rho_i}(-Y) & 
\text{if $\be_i = \rho_i - \al_i < 0$}, \\
0 & \text{otherwise}.
\end{cases}
\]
If $R:=\prod_{i<j} R_{ij}^{n_{ij}}$ is any
raising operator, denote by $\supp(R)$ the set of all
indices $i$ and $j$ such that $n_{ij}>0$. Set $\nu:=R\al$, and define
\[
R \star {}^\rho\wh{c}^\be_{\al} = {}^\rho\ov{c}^\be_{\nu} :=
{}^{\rho_1}\ov{c}_{\nu_1}^{\be_1}\cdots{}^{\rho_\ell}\ov{c}^{\be_\ell}_{\nu_\ell}
\]
where for each $i\geq 1$, 
\[
{}^{\rho_i}\ov{c}_{\nu_i}^{\be_i}:= 
\begin{cases}
{}^{\rho_i}c_{\nu_i}^{\be_i} & \text{if $i\in\supp(R)$}, \\
{}^{\rho_i}\wh{c}_{\nu_i}^{\be_i} & \text{otherwise}.
\end{cases}
\]
If $\al$ is a partition of length $\ell$, then we set
\[
{}^\rho\wh{P}^\be_{\al} := 2^{-\ell}\, \RR \star {}^\rho\wh{c}^\be_{\al}.
\]

\begin{lemma}[\cite{T7}, Lemma 4.9]
\label{Dlem2}
Suppose that $\be_i= \rho_i-\al_i<0$ for every $i\in [1,\ell]$, 
and $\al_j=\al_{j+1}$ and $\be_j=\be_{j+1}$ for some $j\in [1,\ell-1]$.
Then we have ${}^\rho\wh{P}^\be_{\al} = 0$.
\end{lemma}

\subsection{The shape of an element of $\wt{W}_\infty$}
\label{selem}
We next define certain statistics of a signed permutation in
$\wt{W}_\infty$, analogous to the ones given in \S \ref{ssp}.

\begin{defn}
\label{codedefD}
Let $w\in \wt{W}_\infty$. The strict partition $\mu(w)$ is the one
whose parts are the absolute values of the negative entries of $w$
minus one, arranged in decreasing order. The {\em A-code} of $w$ is
the sequence $\gamma=\gamma(w)$ with
$\gamma_i:=\#\{j>i\ |\ w_j<w_i\}$. The parts of the partition
$\delta(w)$ are the non-zero entries $\gamma_i$ arranged in weakly
decreasing order, and $\nu(w):=\delta(w)'$. The {\em shape} of $w$ is
the partition $\la(w):=\mu(w)+\nu(w)$.
\end{defn}

Note that $w$ is uniquely determined by $\mu(w)$ and
$\gamma(w)$, and that $|\la(w)|=\ell(w)$. 

\begin{example}
(a) For the signed permutation $w := (\ov{3}, 2, \ov{7}, \ov{1}, 5, 4,
  \ov{6})$ in $\wt{W}_7$, we obtain $\mu = (6,5,2)$, $\gamma=(2,3, 0,
1, 2, 1, 0)$, $\delta = (3,2,2,1,1)$, $\nu= (5, 3, 1)$, and $\la =
(11, 8, 3)$.

\smallskip \noin (b) Recall from \cite[\S 4.2]{T5} that an element $w$
of $\wt{W}_\infty$ is $n$-Grassmannian if $\ell(ws_i)>\ell(w)$ for all
$i\neq n$. The type of an $n$-Grassmannian element $w$ is 0 if
$|w_1|=1$, and 1 (respectively, 2) if $w_1>1$ (respectively, if
$w_1<-1$). According to \cite[\S 6.1]{BKT1}, there is a type
preserving bijection between the $n$-Grassmannian elements of
$\wt{W}_\infty$ and typed $n$-strict partitions. If $w$ is an
$n$-Grassmannian element of $\wt{W}_\infty$ of type 0 or 1, then
$\la(w)$ is the (typed) $n$-strict partition associated to $w$, in the
sense of op.\ cit.  However, this latter property can fail if
$w_1<-1$, for example the $2$-Grassmannian element
$v:=(\ov{3},4,\ov{1},2)$ is associated to the typed partition of shape
$(2,2)$, while $\la(v)=(3,1)$.
\end{example}

Let $\beta(w)$ denote the sequence defined by $\be(w)_i=-\mu(w)_i$ for
each $i\geq 1$. Recall that $\wt{w}_0^{(n)}$ denotes the longest
element in $\wt{W}_n$. 

\begin{prop}
\label{PfDprop}
Suppose that $m>n\geq 0$ and $w\in \wt{W}_m$ is an $n$-Grassmannian element. 
Set $\wh{w}:=w\wt{w}_0^{(n)}$. Then we have 
\[
\DS_{\wh{w}}(X,Y) = {}^{\nu(\wh{w})}\wh{P}_{\la(\wh{w})}^{\beta(\wh{w})}
\]
in the ring $\Gamma[X_n,Y_{m-1}]$.  In particular, if $w\in S_m$, then
we have
\[
\DS_{\wh{w}}(X,Y) = {}^{\delta_{n-1}}\wh{P}_{2\delta_{n-1}+\la(w)'}^{(1-w_n,\ldots,1-w_1)}.
\]
\end{prop}
\begin{proof}
We first consider the case where $w\in S_m$. We have
\[
w = (a_1,\ldots,a_n,d_1,\ldots, d_r)
\]
where $r=m-n$, $0<a_1<\cdots<a_n$ and $0<d_1<\cdots<d_r$. If $\la:=\la(w)$ then
\[
\la_j=n+j-d_j = m-d_j-(r-j) \ \ \ \text{for $1\leq j \leq r$}.
\]

We have $\wt{w}_0^{(m)}=\wh{w} v_1\cdots v_r$, where
$\ell(\wt{w}_0^{(m)})=\ell(\wh{w}) +\sum_{j=1}^r\ell(v_j)$ and
\[
v_j:=s_{n+j-1}\cdots s_3s_2s_\Box s_1s_2\cdots s_{d_j-1}, \ \ 2\leq j \leq r,
\]
while 
\[
v_1 := \begin{cases}
s_n\cdots s_3s_2s_\Box s_1s_2\cdots s_{d_1-1} & \text{if $d_1>1$}, \\
s_n\cdots s_2s_1 & \text{if $d_1=1$}.
\end{cases}
\]
Using \cite[Thm.\ 1.2]{IMN1} and \cite[Prop.\ 4.10]{T7}, it follows that
\begin{equation}
\label{Dwn}
\DS_{\wh{w}} = \partial_{v_1}\cdots
\partial_{v_r}\left(\DS_{\wt{w}_0^{(m)}}\right) =
\partial_{v_1}\cdots \partial_{v_r} \left(
        {}^{\delta_{m-1}}\wh{P}_{2\delta_{m-1}}^{-\delta_{m-1}}\right).
\end{equation}

According to Lemmas \ref{ddylemD} and \ref{Dlem1dual}, for any $p,q\in
\Z$ with $p\geq 2$, we have
\begin{equation}
\label{pqeqD}
\partial_p({}^p\wh{c}^{-p}_q) = {}^{p-1}\wh{c}^{-p}_{q-1} = {}^p\wh{c}^{1-p}_{q-1} 
-(x_p+y_p)\,{}^{p-1}\wh{c}^{1-p}_{q-2}.
\end{equation}
Let $\epsilon_j$ denote the $j$-th standard basis vector in $\Z^m$. The
Leibnitz rule and (\ref{pqeqD}) imply that for any integer vector
$\al = (\al_1,\ldots,\al_m)$, we have 
\[
\partial_p\left({}^{\delta_{m-1}}\wh{c}^{-\delta_{m-1}}_{\al} \right)= 
{}^{\delta_{m-1}}\wh{c}^{-\delta_{m-1}+\epsilon_{m-p}}_{\al-\epsilon_{m-p}}-
(x_p+y_p)\,{}^{\delta_{m-1}-\epsilon_{m-p}}
\wh{c}^{-\delta_{m-1}+\epsilon_{m-p}}_{\al - 2\epsilon_{m-p}}.
\]
We deduce from this and Lemma \ref{Dlem2} that 
\begin{align*}
\partial_p {}^{\delta_{m-1}}\wh{P}_{2\delta_{m-1}}^{-\delta_{m-1}}&=
        {}^{\delta_{m-1}}\wh{P}^{-\delta_{m-1}+\epsilon_{m-p}}_{2\delta_{m-1}-\epsilon_{m-p}}
        \\ &= {}^{(m-1,\ldots,1)} \wh{P}_{(2m-2,\ldots,2p+2, 2p-1, 2p-2,2p-4,\ldots,2)}^{(1-m,\ldots, -1-p, 1-p,1-p,2-p,\ldots,-1)}.
\end{align*}
Iterating this calculation for $p=d_r-1,\ldots,2$ gives
\begin{align*}
(\partial_2\cdots\partial_{d_r-1})\DS_{\wt{w}_0^{(m)}} &= 
{}^{(m-1,\ldots,1)}
\wh{P}_{(2m-2,\ldots,2d_r, 2d_r-3, 2d_r-5,\ldots,3,2)}^{(1-m,\ldots, -d_r, 2-d_r,3-d_r,\ldots, -1,-1)} \\
&={}^{(m-1,\ldots,1)}\wh{P}_{(2m-2,\ldots,2d_r, 2d_r-3, 2d_r-5,\ldots,3,2)}^{(1-m,\ldots, -d_r, 2-d_r,3-d_r,\ldots, -1,-1)}.
\end{align*}

We next compute that $$\partial_\Box\partial_1({}^1\wh{c}^{-1}_p) = 
\begin{cases} 2 & \text{if $p=2$}, \\ 0 & \text{otherwise}.
\end{cases}$$ By arguing as in \cite[Prop.\ 4.10]{T7}, it follows that
\[
(\partial_\Box\partial_1\cdots\partial_{d_r-1})\DS_{\wt{w}_0^{(m)}} =
{}^{(m-1,\ldots,2)}\wh{P}_{(2m-2,\ldots,2d_r, 2d_r-3, 2d_r-5,\ldots,3)}^{(1-m,\ldots, -d_r, 2-d_r,3-d_r,\ldots, -1)}.
\]
Applying Lemma \ref{ddylemD} alone $m-2$ times to the last equality gives
\[
\partial_{v_r}(\DS_{\wt{w}_0^{(m)}}) = {}^{\delta_{m-2}}\wh{P}^{(1-m,\ldots,
  -d_r, 2-d_r,3-d_r,\ldots,
  -1)}_{(2m-3,\ldots,2d_r-1,2d_r-4,\ldots,2)} =
        {}^{\delta_{m-2}}\wh{P}^{(1-m,\ldots,\wh{1-d_r},\ldots,-1)}_{2\delta_{m-2}+1^{m-d_r}}.
\]
We now use (\ref{Dwn}) and repeat the above computation $r-1$ more
times to get
\[
\DS_{\wh{w}} = {}^{\delta_{n-1}}\wh{P}_{2\delta_{n-1}+\xi}^{\rho}
\]
where 
\[
\rho = (1-m,\ldots,\wh{1-d_r},\ldots,\wh{1-d_1},\ldots,-1) = 
(1-w_n,\ldots,1-w_1)
\]
and 
\[
\xi = \sum_{j=1}^r1^{m-d_j-(r-j)} = \sum_{j=1}^r 1^{\la_j} = \la(w)'.
\]
(Note that in the case when $d_1=1$, the last stage of the calculation is 
simpler).

The general case now follows as in the proof of Proposition 
\ref{PfCprop}. 
\end{proof}

\subsection{Double eta polynomials and alternating sums}
Let $\la$ be a typed $n$-strict partition which corresponds to the
$n$-Grassmannian element $w \in \wt{W}_\infty$, and define a sequence
$\beta(\la)$ and a set $\cC(\la)$ using the same formulas
(\ref{betaC}) and (\ref{CC}) as in Lie type C.  The {\em double eta
  polynomial} ${}^n\Eta_\la(X,Y)$ of \cite{T6} is defined by
\[
{}^n\Eta_\la(X,Y) := 2^{-\ell_n(\la)}\prod_{i<j} (1-R_{ij}) \prod_{(i,j)\in \cC(\la)}
(1+R_{ij})^{-1}\star({}^n\wh{c})^{\be(\la)}_\la
\]
where $\ell_n(\la)$ denotes the number of parts $\la_i$ which are
greater than $n$ (see op.\ cit.\ for the precise definitions of 
typed $n$-strict partitions and $\star$).

Let $\A':\Gamma'[X_n,Y]\to \Gamma'[X_n,Y]$ be the operator defined by
\[
\A'(f):=\sum_{w\in \wt{W}_n}(-1)^{\ell(w)}w(f).
\]
Let $\wt{w}_0$ denote the longest element in $\wt{W}_n$ and set 
$\wh{w}:=w\wt{w}_0$. 

\begin{thm}
Let $\la$ be a typed $n$-strict partition and $w$ be the corresponding
$n$-Grassmannian element of $\wt{W}_\infty$. Then we have
\begin{align}
\label{EtoPdoub}
{}^n\Eta_{\la}(X,Y) &= \partial_{\wt{w}_0}\left({}^{\nu(\wh{w})}
\wh{P}^{\be(\wh{w})}_{\la(\wh{w})}\right) \\
\label{EtoP2doub}
&=(-1)^{n(n-1)/2}\cdot 2^{n-1} \left. \A'\left({}^{\nu(\wh{w})}
\wh{P}_{\la(\wh{w})}^{\be(\wh{w})}\right)\right\slash
\A'\left(x^{2\delta_{n-1}}\right).
\end{align}
\end{thm}
\begin{proof}
We deduce from (\ref{Dddeq}) that 
the double Schubert polynomial $\DS_w(X,Y)$ satisfies
\begin{equation}
\label{DtoD}
\DS_w(X,Y)=\partial_{\wt{w}_0} \left(\DS_{\wh{w}}(X,Y)\right).
\end{equation}
The equality (\ref{EtoPdoub}) follows from (\ref{DtoD}), Proposition
\ref{PfDprop}, and the fact, proved in \cite{T6}, that 
$\DS_w(X,Y) = {}^n\Eta_{\la}(X,Y)$ in the ring $\Gamma'[X_n,Y]$.

For the second equality, recall from \cite[Lemma 4]{D} 
and \cite{PR} that we have 
\[
\partial_{\wt{w}_0}(f) = (-1)^{n(n-1)/2}
\prod_{1\leq i < j\leq n}(x_i^2-x_j^2)^{-1}\cdot \A'(f).
\]
On the other hand, it follows from \cite[Lemma 5.16(ii)]{PR} that
\[
\partial_{\wt{w}_0}(x^{2\delta_{n-1}}) = (-1)^{n(n-1)/2} \cdot 2^{n-1}
\]
and hence that 
\[
\A'(x^{2\delta_{n-1}}) = 2^{n-1}\prod_{1\leq i < j\leq n}(x_i^2-x_j^2).
\]
The proof of (\ref{EtoP2doub}) is completed by using these two
equations in (\ref{EtoPdoub}).
\end{proof}

\section{Single Schubert polynomials of type D}
\label{sspD}

\subsection{Eta polynomials as Weyl group invariants}
In this section, we work with the single type D Schubert polynomials
$\DS_w(X)$. Let $\chi':\Gamma'[X_n]\to \Z$ be the homomorphism defined by 
$\chi'(b_p)=\chi'(x_j)=0$ for all $p,j$. 

\begin{prop}
For any $f\in \Gamma'[X_n]$, we have 
$f=\sum_{w\in \wt{W}^{(n)}}\chi'(\partial_wf)\DS_w(X_n)$.
\end{prop}
\begin{proof}
The argument is the same as for the proof of Proposition \ref{etaprop}.
\end{proof}

We next define a ring $B^{(n)}$, following \cite[\S 5.2]{BKT2}.
For each integer $p\geq 1$, let
\[
{}^nb_p:=\begin{cases} 
e_p(X_n) + 2\sum_{j=0}^{p-1} b_{p-j}e_j(X_n) & \text{if $p<n$}, \\
\sum_{j=0}^p b_{p-j}e_j(X_n) & \text{if $p\geq n$}
\end{cases}
\]
and 
\[
{}^nb'_n:=\sum_{j=0}^{n-1}b_{n-j}e_j(X_n).
\]
Observe that the elements denoted by $\eta_r(x\,;y)$ and
$\eta'_k(x\,;y)$ in loc.\ cit.\ correspond to the elements ${}^nb_r$
and ${}^nb'_n$ here. Let
\[
B^{(n)}:= \Z[{}^nb_1, \ldots, {}^nb_{n-1},{}^nb_n, {}^nb'_n,
{}^nb_{n+1}, \ldots]
\]
be the ring of eta polynomials of level $n$. We have
\[
{}^nc_p=\begin{cases}
{}^nb_p & \text{if $p<n$}, \\
{}^nb_n+{}^nb'_n & \text{if $p=n$}, \\
2\cdot {}^nb_p & \text{if $p>n$}
\end{cases}
\]
and thus $\Gamma^{(n)}$ is a subring of $B^{(n)}$.

According to \cite[Thm.\ 4]{BKT2}, the single eta polynomials
${}^n\Eta_\la$ for all typed $n$-strict partitions $\la$ form a
$\Z$-basis of $B^{(n)}$. In the next result, the Weyl group $\wt{W}_n$
acts on the ring $\Gamma'[X_n]$ in the manner described in \S
\ref{Dprelims}.

\begin{prop}
\label{invlemD}
  The ring $B^{(n)}$ is equal to the subring $\Gamma'[X_n]^{\wt{W}_n}$ of
  $\wt{W}_n$-invariants in $\Gamma'[X_n]$.
\end{prop}
\begin{proof}
  The proof is identical to that of Proposition \ref{invlem}, using
  \cite[Prop.\ 6.3]{BKT2} for the fact that the Schubert polynomials
  $\DS_w(X_n)$ for those $w\in \wt{W}_\infty$ with $|w_1|<w_2<\cdots < w_n$
  and $w_{n+1}<w_{n+2}<\cdots$ coincide with the (single) eta
  polynomials of level $n\geq 2$.
\end{proof}

\begin{example}
It follows from Proposition \ref{invlemD} that we have
\[
B^{(n)}\cap \Z[X_n] = \Z[X_n]^{\wt{W}_n} = 
\Z[e_1(X^2_n),\ldots,e_{n-1}(X^2_n),e_n(X_n)].
\]
This can also be shown as in Example \ref{Cex}, using the fact that 
${}^nb_n-{}^nb_n' = e_n(X_n)$.
\end{example}

For any parabolic subgroup $P$ of $\SO_{2n}$, let
$\wt{W}_n^P:=\wt{W}^P\cap \wt{W}_n$.  Let $\IB^{(n)}$ (respectively
$\IB^{(n)}_P$) be the ideal of $\Gamma'[X_n]_\Q$ (respectively
$\Gamma'[X_n]_\Q^{\wt{W}_P}$) generated by the homogeneous elements in
$B^{(n)}$ of positive degree. We then have the following immediate
consequence of Theorem \ref{Dthm} and the discussion in \S
\ref{Dgeom}.

\begin{cor}
There is a canonical ring isomorphism 
\[
\HH^*(\SO_{2n}/B,\Q) \cong \Gamma'[X_n]_\Q/\IB^{(n)}
\]
which maps the cohomology class of the codimension $\ell(w)$ Schubert
variety $X_w$ to the class of the Schubert polynomial $\DS_w(X)$, for
any $w\in \wt{W}_n$. Moreover, for any parabolic subgroup $P$ of
$\SO_{2n}$, there is a canonical ring isomorphism
\[
\HH^*(\SO_{2n}/P,\Q) \cong \Gamma'[X_n]_\Q^{\wt{W}_P}/\IB^{(n)}_P
\]
which maps the cohomology class of the codimension $\ell(w)$ Schubert
variety $X_w$ to the class of the Schubert polynomial $\DS_w(X)$, for
any $w\in \wt{W}^P_n$.
\end{cor}

\subsection{The ring $\Gamma'[X_n]_\Q$ as a $B^{(n)}_\Q$-module}
Let $e_p:=e_p(X_n)$ for each $p\in \Z$, and recall that $\cP_n$
denotes the set of strict partitions $\la$ with $\la_1\leq n$.

\begin{prop}
\label{freemodold}
The $\Q$-algebra $\Gamma'[X_n]_\Q$ is a free $B^{(n)}_\Q$-module of 
rank $2^{n-1}n!$ with basis
\[
\{e_\la(-X_n)x^\al\ |\ \la \in \cP_{n-1}, 
\ \ 0\leq \al_i \leq n-i, \ i\in [1,n]\}.
\]
\end{prop}
\begin{proof}
We have that $\Gamma'[X_n]$ is a free
$\Gamma'[e_1,\ldots,e_n]$-module with basis given by the monomials
$x^\al$ with $0\leq \al_i \leq n-i$ for  $i\in [1,n]$. It will therefore 
suffice to show that $\Gamma'[e_1,\ldots,e_n]_\Q$ is a free $B^{(n)}_\Q$-module
with basis $e_\la(-X_n)$ for $\la\in \cP_{n-1}$.

As in the proof of Proposition \ref{freemod}, we see that the monomials 
$e_\la(-X_n)$ for
$\la\in\cP_n$ generate $\Gamma'[e_1,\ldots,e_n]_\Q$ as a 
$B^{(n)}_\Q$-module. Furthermore, since $e_n={}^nb_n-{}^nb_n'$, it follows 
that the monomials $e_\la(-X_n)$ for $\la\in \cP_{n-1}$ also generate 
this module. The rest of the argument is the same as in type C.
\end{proof}

For any strict partition $\la$,
define the $\wt{P}$-polynomial of \cite{PR} by
\[
\wt{P}_\la(X_n):=2^{-\ell(\la)}\,\wt{Q}_\la(X_n).
\]

\begin{cor}
The ring $\Gamma'[X_n]$ is a free $\Gamma'[X_n]^{S_n}$-module with
basis $\{\AS_\om(X)\}$ for $\om\in S_n$. The $\Q$-algebra
$\Gamma'[X_n]^{S_n}_\Q$ is a free $B^{(n)}_\Q$-module with basis
$\{P_\la(-X_n)\}$ for $\la\in\cP_{n-1}$.  The $\Q$-algebra
$\Gamma'[X_n]_\Q$ is a free $B^{(n)}_\Q$-module on the basis
$\{\DS_w(X_n)\}$ of single type D Schubert polynomials for $w\in
\wt{W}_n$, and is also free on the product basis
$\{P_\la(-X_n)\AS_\om(X)\}$ for $\la\in\cP_{n-1}$ and $\om\in S_n$.
\end{cor}

\subsection{A scalar product on $\Gamma'[X_n]$}
Let $\wt{w}_0$ be the element of longest length in $\wt{W}_n$.  If
$f\in \Gamma'[X_n]$, then $\partial_i(\partial_{\wt{w}_0}f)= 0$ for
all $i$ with $\Box\leq i \leq n-1$. Hence Proposition \ref{invlemD}
implies that $\partial_{\wt{w}_0}(f)\in B^{(n)}$, for each $f\in
\Gamma'[X_n]$.

\begin{defn}
We define a scalar product $\langle\ ,\, \rangle$
on $\Gamma'[X_n]$, with values in $B^{(n)}$,
by the rule 
\[
\langle f,g\rangle := \partial_{\wt{w}_0}(fg), \ \ \ f,g\in \Gamma'[X_n].
\]
\end{defn}

\begin{prop}
\label{2propertiesD}
{\em (a)} The scalar product $\langle\ ,\, \rangle: \Gamma'[X_n]\times \Gamma'[X_n]
\to B^{(n)}$ is $B^{(n)}$-linear. For any $f,g\in \Gamma'[X_n]$ and 
$w\in \wt{W}_n$, we have
\[
\langle\partial_wf,g \rangle  = \langle f,\partial_{w^{-1}}g\rangle.
\]
\medskip \noin {\em (b)} Let $u,v\in \wt{W}_n$ be such that
$\ell(u)+\ell(v)=n^2-n$. Then we have
\[
\langle\DS_u(X_n),\DS_v(X_n)\rangle = 
\begin{cases}
1 & \text{if $v=\wt{w}_0u$}, \\
0 & \text{otherwise}.
\end{cases}
\]
\end{prop}
\begin{proof}
The argument is identical to the proofs of Propositions \ref{2properties}
and \ref{prodprop}.
\end{proof}

Let $\om_0$ denote the longest permutation 
in $S_n$, and define $\wt{v}_0:=\wt{w}_0\om_0$.  We define a
$\Gamma'[X_n]^{S_n}$-valued scalar product $(\ ,\, )$ on
$\Gamma'[X_n]$ by the rule
\[
(f,g) := \partial_{\om_0}(fg), \ \ \ f,g\in \Gamma'[X_n],
\]
and a $B^{(n)}$-valued scalar product $\{\ ,\, \}$
on $\Gamma'[X_n]^{S_n}$ by the rule 
\[
\{f,g\} := \partial_{\wt{v}_0}(fg), \ \ \ f,g\in \Gamma'[X_n]^{S_n}.
\]
Following \cite[Thm.\ 5.23]{PR}, for any two partitions $\la,\mu\in \cP_{n-1}$, 
we have 
\[
\left\{\wt{P}_\la(-X_n),\wt{P}_{\delta_{n-1}\ssm \mu}(-X_n)\right\} = \delta_{\la,\mu},
\]
where $\delta_{n-1}\ssm \mu$ is the strict partition whose parts complement
the parts of $\mu$ in the set $\{n-1,n-2,\ldots,1\}$

Observe that the scalar product $(\ ,\, )$ is
$\Gamma'[X_n]^{S_n}$-linear and $\{\ ,\, \}$ is
$B^{(n)}$-linear. Since $\partial_{\wt{w}_0} =
\partial_{\wt{v}_0}\partial_{\om_0}$, we deduce that
\[
\langle f,g \rangle = \{(f,g)\}, \ \ \ \text{for any $f,g\in\Gamma[X_n]$}.
\]
Furthermore, according to \cite[(2.20)]{LP2}, the orthogonality 
relation
\[
\left\langle\wt{P}_\la(-X_n)\AS_u(X), \wt{P}_{\delta_{n-1}\ssm \mu}(-X_n)
(\om_0\AS_{u'\om_0}(-X)) \right\rangle =\delta_{u,u'}\delta_{\la,\mu}
\]
holds, for any $u,u'\in S_n$ and $\la,\mu\in \cP_{n-1}$. We have therefore
identified the dual $B^{(n)}_\Q$-basis of the product basis
$\{\wt{P}_\la(-X_n)\AS_u(X)\}$ of $\Gamma'[X_n]_\Q$, relative to the
scalar product $\langle\ ,\, \rangle$.

\medskip
\medskip
\noindent
{\em Acknowledgement.} I thank Sara Billey, Tom Haines, and Andrew
Kresch for their helpful comments, and the anonymous referee for a
careful reading of the paper and suggestions with helped to clarify
the exposition.

\end{document}